\documentclass[11pt,oneside,english,reqno]{amsart}
\usepackage[T1]{fontenc}
\usepackage[latin9]{inputenc}
\usepackage{geometry}
\usepackage{babel}
\usepackage{float}
\usepackage{mathrsfs}
\usepackage{amsthm}
\usepackage{amstext}
\usepackage{amssymb}
\usepackage{graphicx}
\usepackage{setspace}
\usepackage{esint}
\PassOptionsToPackage{normalem}{ulem}
\usepackage{ulem}
\usepackage[unicode=true,pdfusetitle,
 bookmarks=true,bookmarksnumbered=false,bookmarksopen=false,
 breaklinks=false,pdfborder={0 0 1},backref=false,colorlinks=false]
 {hyperref}
\hypersetup{
 colorlinks=true,citecolor=blue,linkcolor=blue,linktocpage=true}

\makeatletter

\providecommand{\tabularnewline}{\\}

\numberwithin{equation}{section}
\numberwithin{figure}{section}
\numberwithin{table}{section}
\usepackage{enumitem}		
\theoremstyle{plain}
\newtheorem{thm}{\protect\theoremname}[section]
  \theoremstyle{definition}
  \newtheorem{defn}[thm]{\protect\definitionname}
  \theoremstyle{remark}
  \newtheorem{rem}[thm]{\protect\remarkname}
  \theoremstyle{plain}
  \newtheorem{lem}[thm]{\protect\lemmaname}
  \theoremstyle{plain}
  \newtheorem{cor}[thm]{\protect\corollaryname}
  \theoremstyle{remark}
  \newtheorem*{claim*}{\protect\claimname}
  \theoremstyle{definition}
  \newtheorem{example}[thm]{\protect\examplename}
  \theoremstyle{plain}
  \newtheorem{prop}[thm]{\protect\propositionname}
  \theoremstyle{plain}
  \newtheorem{question}[thm]{\protect\questionname}
  \theoremstyle{plain}
  \newtheorem*{fact*}{\protect\factname}
  \theoremstyle{remark}
  \newtheorem*{acknowledgement*}{\protect\acknowledgementname}

\allowdisplaybreaks

\providecommand{\MR}[1]{}

\renewcommand{\section}{%
\@startsection{section}{1}%
  \z@{.7\linespacing\@plus\linespacing}{.5\linespacing}%
  {\normalfont\scshape\centering\bfseries}}

\renewcommand{\subsection}{%
\@startsection{subsection}{2}%
  \z@{.5\linespacing\@plus.7\linespacing}{.5\linespacing}%
  {\normalfont\bfseries}}

\renewcommand{\subsubsection}{%
\@startsection{subsubsection}{2}%
  \z@{.5\linespacing\@plus.7\linespacing}{.5\linespacing}%
  {\normalfont\bfseries}}

\AtBeginDocument{
  
}

\makeatother

  \providecommand{\acknowledgementname}{Acknowledgement}
  \providecommand{\claimname}{Claim}
  \providecommand{\corollaryname}{Corollary}
  \providecommand{\definitionname}{Definition}
  \providecommand{\examplename}{Example}
  \providecommand{\factname}{Fact}
  \providecommand{\lemmaname}{Lemma}
  \providecommand{\propositionname}{Proposition}
  \providecommand{\questionname}{Question}
  \providecommand{\remarkname}{Remark}
\providecommand{\theoremname}{Theorem}

\begin{document}

\title{Discrete reproducing kernel Hilbert spaces: Sampling and distribution
of Dirac-masses}

\author{Palle Jorgensen and Feng Tian}

\address{(Palle E.T. Jorgensen) Department of Mathematics, The University
of Iowa, Iowa City, IA 52242-1419, U.S.A. }

\email{palle-jorgensen@uiowa.edu}

\urladdr{http://www.math.uiowa.edu/\textasciitilde{}jorgen/}

\address{(Feng Tian) Department of Mathematics, Wright State University, Dayton,
OH 45435, U.S.A.}

\email{feng.tian@wright.edu}

\urladdr{http://www.wright.edu/\textasciitilde{}feng.tian/}



\maketitle
\pagestyle{myheadings}
\markright{}
\begin{abstract}
We study reproducing kernels, and associated reproducing kernel Hilbert
spaces (RKHSs) $\mathscr{H}$ over infinite, discrete and countable
sets $V$. In this setting we analyze in detail the distributions
of the corresponding Dirac point-masses of $V$. Illustrations include
certain models from neural networks: An Extreme Learning Machine (ELM)
is a neural network-configuration in which a hidden layer of weights
are randomly sampled, and where the object is then to compute resulting
output. For RKHSs $\mathscr{H}$ of functions defined on a prescribed
countable infinite discrete set $V$, we characterize those which
contain the Dirac masses $\delta_{x}$ for all points $x$ in $V$.
Further examples and applications where this question plays an important
role are: (i) discrete Brownian motion-Hilbert spaces, i.e., discrete
versions of the Cameron-Martin Hilbert space; (ii) energy-Hilbert
spaces corresponding to graph-Laplacians where the set $V$ of vertices
is then equipped with a resistance metric; and finally (iii) the study
of Gaussian free fields.
\end{abstract}

\tableofcontents{}

\section{Introduction}

A reproducing kernel Hilbert space (RKHS) is a Hilbert space $\mathscr{H}$
of functions on a prescribed set, say $V$, with the property that
point-evaluation for functions $f\in\mathscr{H}$ is continuous with
respect to the $\mathscr{H}$-norm. They are called kernel spaces,
because, for every $x\in V$, the point-evaluation for functions $f\in\mathscr{H}$,
$f\left(x\right)$ must then be given as a $\mathscr{H}$-inner product
of $f$ and a vector $k_{x}$, in $\mathscr{H}$; called the kernel.

The RKHSs have been studied extensively since the pioneering papers
by Aronszajn in the 1940ties, see e.g., \cite{Aro43,Aro48}. They
further play an important role in the theory of partial differential
operators (PDO); for example as Green's functions of second order
elliptic PDOs; see e.g., \cite{Nel57,HKL14}. Other applications include
engineering, physics, machine-learning theory (see \cite{KH11,SZ09,CS02}),
stochastic processes (e.g., Gaussian free fields), numerical analysis,
and more. See, e.g., \cite{AD93,ABDdS93,AD92,AJSV13,AJV14}. Also,
see \cite{MR2089140,MR2607639,MR2913695,MR2975345,MR3091062,MR3101840,MR3201917}.
But the literature so far has focused on the theory of kernel functions
defined on continuous domains, either domains in Euclidean space,
or complex domains in one or more variables. For these cases, the
Dirac $\delta_{x}$ distributions do not have finite $\mathscr{H}$-norm.
But for RKHSs over discrete point distributions, it is reasonable
to expect that the Dirac $\delta_{x}$ functions will in fact have
finite $\mathscr{H}$-norm.

An illustration from neural networks: An Extreme Learning Machine
(ELM) is a neural network configuration in which a hidden layer of
weights are randomly sampled (see e.g., \cite{RW06}), and the object
is then to determine analytically resulting output layer weights.
Hence ELM may be thought of as an approximation to a network with
infinite number of hidden units.

Here we consider the discrete case, i.e., RKHSs of functions defined
on a prescribed countable infinite discrete set $V$. We are concerned
with a characterization of those RKHSs $\mathscr{H}$ which contain
the Dirac masses $\delta_{x}$ for all points $x\in V$. Of the examples
and applications where this question plays an important role, we emphasize
three: (i) discrete Brownian motion-Hilbert spaces, i.e., discrete
versions of the Cameron-Martin Hilbert space; (ii) energy-Hilbert
spaces corresponding to graph-Laplacians; and finally (iii) RKHSs
generated by binomial coefficients. We show that the point-masses
have finite $\mathscr{H}$-norm in cases (i) and (ii), but not in
case (iii).

Our setting is a given positive definite function $k$ on $V\times V$,
where $V$ is discrete (see above). We study the corresponding RKHS
$\mathscr{H}\left(=\mathscr{H}\left(k\right)\right)$ in detail. Our
main results are Theorems \ref{thm:del}, \ref{thm:bino}, and \ref{thm:mc}
which give explicit answers to the question of which point-masses
from $V$ are in $\mathscr{H}$. Applications include Corollaries
\ref{cor:proj}, \ref{cor:bino}, \ref{cor:lap1}, \ref{cor:lap2},
\ref{cor:lap3}, and \ref{cor:Lap3}.

The paper is organized as follows: Section \ref{sec:drkhs} leads
up to our characterization (Theorem \ref{thm:del}) of point-masses
which have finite $\mathscr{H}$-norm. It is applied in sections \ref{sec:egs}
and \ref{sec:net} to a variety of classes of discrete RKHSs. Section
\ref{sec:egs} deals with samples from Brownian motion, and from the
Brownian bridge process, and binomial kernels, and with kernels on
sets $V\times V$ which arise as restrictions to sample-points. Section
\ref{sec:net} covers the case of infinite network of resistors. By
this we mean an infinite graph with assigned resistors on its edges.
In this family of examples, the associated RKHSs vary with the assignment
of resistors on the edges in $G$, and are computed explicitly from
a resulting energy form. Our result Corollary \ref{cor:lap1} states
that, for the network models, all point-masses have finite energy.
Furthermore, we compute the value, and we study $V$ as a metric space
w.r.t. the corresponding resistance metric. These results, in turn,
have direct implications (Corollaries \ref{cor:lap2}, \ref{cor:lap3}
and \ref{cor:lap}) for the family of Gaussian free fields associated
with our infinite network models.

A positive definite kernel $k$ is said to be \emph{universal} \cite{CMPY08}
if, every continuous function, on a compact subset of the input space,
can be uniformly approximated by sections of the kernel, i.e., by
continuous functions in the RKHS. In Theorem \ref{thm:mc} we show
that for the RKHSs from kernels $k_{c}$ in electrical network $G$
of resistors, this universality holds. The metric in this case is
the resistance metric on the vertices of $G$, determined by the assignment
of a conductance function $c$ on the edges in $G$.

\section{\label{sec:drkhs}Discrete RKHSs}
\begin{defn}
Let $V$ be a countable and infinite set, and $\mathscr{F}\left(V\right)$
the set of all \emph{finite} subsets of $V$. A function $k:V\times V\rightarrow\mathbb{C}$
is said to be \emph{positive definite}, if 
\begin{equation}
\underset{\left(x,y\right)\in F\times F}{\sum\sum}k\left(x,y\right)\overline{c_{x}}c_{y}\geq0\label{eq:pd1}
\end{equation}
holds for all coefficients $\{c_{x}\}_{x\in F}\subset\mathbb{C}$,
and all $F\in\mathscr{F}\left(V\right)$. 
\end{defn}

\begin{defn}
\label{def:d1}Fix a set $V$, countable infinite. 
\begin{enumerate}
\item For all $x\in V$, set 
\begin{equation}
k_{x}:=k\left(\cdot,x\right):V\rightarrow\mathbb{C}\label{eq:pd2}
\end{equation}
as a function on $V$. 
\item Let $\mathscr{H}:=\mathscr{H}\left(k\right)$ be the Hilbert-completion
of the $span\left\{ k_{x}:x\in V\right\} $, with respect to the inner
product 
\begin{equation}
\left\langle \sum c_{x}k_{x},\sum d_{y}k_{y}\right\rangle _{\mathscr{H}}:=\sum\sum\overline{c_{x}}d_{y}k\left(x,y\right)\label{eq:pd3}
\end{equation}
modulo the subspace of functions of zero $\mathscr{H}$-norm. $\mathscr{H}$
is then a reproducing kernel Hilbert space (HKRS), with the reproducing
property:
\begin{equation}
\left\langle k_{x},\varphi\right\rangle _{\mathscr{H}}=\varphi\left(x\right),\;\forall x\in V,\:\forall\varphi\in\mathscr{H}.\label{eq:pd31}
\end{equation}
\textbf{Note.} The summations in (\ref{eq:pd3}) are all finite. Starting
with finitely supported summations in (\ref{eq:pd3}), the RKHS $\mathscr{H}=\mathscr{H}\left(k\right)$
is then obtained by Hilbert space completion. We use physicists' convention,
so that the inner product is conjugate linear in the first variable,
and linear in the second variable.
\item If $F\in\mathscr{F}\left(V\right)$, set $\mathscr{H}_{F}=\text{closed\:\ span}\{k_{x}\}_{x\in F}\subset\mathscr{H}$,
(closed is automatic if $F$ is finite.) And set 
\begin{equation}
P_{F}:=\text{the orthogonal projection onto \ensuremath{\mathscr{H}_{F}}}.\label{eq:pd4}
\end{equation}

\item For $F\in\mathscr{F}\left(V\right)$, set 
\begin{equation}
K_{F}:=\left(k\left(x,y\right)\right)_{\left(x,y\right)\in F\times F}\label{eq:pd5}
\end{equation}
as a $\#F\times\#F$ matrix. 
\end{enumerate}
\end{defn}
\begin{rem}
It follows from the above that reproducing kernel Hilbert spaces (RKHS)
arise from a given positive definite kernel $k$, a corresponding
pre-Hilbert form; and then a Hilbert-completion. The question arises:
\textquotedblleft What are the functions in the completion?\textquotedblright{}
Now, before completion, the functions are as specified in Definition
\ref{def:d1}, but the Hilbert space completions are subtle; they
are classical Hilbert spaces of functions, not always transparent
from the naked kernel $k$ itself. Examples of classical RKHSs: Hardy
spaces or Bergman spaces (for complex domains), Sobolev spaces and
Dirichlet spaces (for real domains, or for fractals \cite{MR3054607,MR2892621,MR2764237}),
band-limited $L^{2}$ functions (from signal analysis), and Cameron-Martin
Hilbert spaces from Gaussian processes (in continuous time domain).

Our focus here is on discrete analogues of the classical RKHSs from
real or complex analysis. These discrete RKHSs in turn are dictated
by applications, and their features are quite different from those
of their continuous counter parts.\end{rem}
\begin{defn}
\label{def:dmp}The RKHS $\mathscr{H}=\mathscr{H}\left(k\right)$
is said to have the \emph{discrete mass} property ($\mathscr{H}$
is called a \emph{discrete RKHS}), if $\delta_{x}\in\mathscr{H}$,
for all $x\in V$. Here, $\delta_{x}\left(y\right)=\begin{cases}
1 & \text{if }x=y\\
0 & \text{if \ensuremath{x\neq y}}
\end{cases}$, i.e., the Dirac mass at $x\in V$. \end{defn}
\begin{lem}
\label{lem:proj1}Let $F\in\mathscr{F}\left(V\right)$, $x_{1}\in F$.
Assume $\delta_{x_{1}}\in\mathscr{H}$. Then 
\begin{equation}
P_{F}\left(\delta_{x_{1}}\right)\left(\cdot\right)=\sum_{y\in F}\left(K_{F}^{-1}\delta_{x_{1}}\right)\left(y\right)k_{y}\left(\cdot\right).\label{eq:pd6}
\end{equation}
\end{lem}
\begin{proof}
Show that
\begin{equation}
\delta_{x}-\sum_{y\in F}\left(K_{F}^{-1}\delta_{x_{1}}\right)\left(y\right)k_{y}\left(\cdot\right)\in\mathscr{H}_{F}^{\perp}.\label{eq:pd7}
\end{equation}
The remaining part follows easily from this. 

(The notation $\left(\mathscr{H}_{F}\right)^{\perp}$ stands for orthogonal
complement, also denoted $\mathscr{H}\ominus\mathscr{H}_{F}=\left\{ \varphi\in\mathscr{H}\:\big|\:\left\langle f,\varphi\right\rangle _{\mathscr{H}}=0,\;\forall f\in\mathscr{H}_{F}\right\} $.)\end{proof}
\begin{lem}
\label{lem:proj}Using Dirac's bra-ket, and ket-bra notation (for
rank-one operators), the orthogonal projection onto $\mathscr{H}_{F}$
is 
\begin{equation}
P_{F}=\sum_{y\in F}\left|k_{y}\left\rangle \right\langle k_{y}^{*}\right|;\label{eq:pd71}
\end{equation}
where 
\begin{equation}
k_{x}^{*}:=\sum_{y\in F}\left(K_{F}^{-1}\right)_{yx}k_{y}\label{eq:pd72}
\end{equation}
is the dual vector to $k_{x}$, for all $x\in V$. \end{lem}
\begin{proof}
Let $k_{x}^{*}$ be specified as in (\ref{eq:pd72}), then 
\begin{eqnarray*}
\left\langle k_{x}^{*},k_{z}\right\rangle _{\mathscr{H}} & = & \sum_{y\in F}\left\langle \left(K_{F}^{-1}\right)_{yx}k_{y},k_{z}\right\rangle _{\mathscr{H}}\\
 & = & \sum_{y\in F}\left(K_{F}^{-1}\right)_{xy}\left\langle k_{y},k_{z}\right\rangle _{\mathscr{H}}\\
 & = & \sum_{y\in F}\left(K_{F}^{-1}\right)_{xy}\left(K_{F}\right)_{yz}=\delta_{x,z},
\end{eqnarray*}
i.e., $k_{x}^{*}$ is the dual vector to $k_{x}$, for all $x\in V$. 

For $f\in\mathscr{H}$, and $F\in\mathscr{F}\left(V\right)$, we have
\begin{eqnarray*}
\sum_{y\in F}\left|k_{y}\left\rangle \right\langle k_{y}^{*}\right|f & = & \sum_{y\in F}\left\langle k_{y}^{*},f\right\rangle _{\mathscr{H}}k_{y}\\
 & = & \underset{\left(y,z\right)\in F\times F}{\sum\sum}\left(K_{F}^{-1}\right)_{z,y}\left\langle k_{z},f\right\rangle _{\mathscr{H}}\\
 & = & P_{F}f.
\end{eqnarray*}
This yields the orthogonal projection realized as stated in (\ref{eq:pd71}). 

Now, applying (\ref{eq:pd71}) to $\delta_{x_{1}}$, we get 
\begin{eqnarray*}
P_{F}\left(\delta_{x_{1}}\right) & = & \sum_{y\in F}\left\langle k_{y}^{*},\delta_{x_{1}}\right\rangle _{\mathscr{H}}k_{y}\\
 & = & \sum_{y\in F}\left(\sum_{z\in F}\left(K_{F}^{-1}\right)_{yz}\left\langle k_{z},\delta_{x_{1}}\right\rangle _{\mathscr{H}}\right)k_{y}\\
 & = & \sum_{y\in F}\left(\sum_{z\in F}\left(K_{F}^{-1}\right)_{yz}\delta_{x_{1}}\left(z\right)\right)k_{y}\\
 & = & \sum_{y\in F}\left(K_{F}^{-1}\delta_{x_{1}}\right)\left(y\right)k_{y},
\end{eqnarray*}
where 
\[
\left(K_{F}^{-1}\delta_{x_{1}}\right)\left(y\right):=\sum_{z\in F}\left(K_{F}^{-1}\right)_{yz}\delta_{x_{1}}\left(z\right).
\]
This verifies (\ref{eq:pd6}). 
\end{proof}

\begin{rem}
Note a slight abuse of notations: We make formally sense of the expressions
for $P_{F}(\delta_{x})$ in (\ref{eq:pd6}) even in the case when
$\delta_{x}$ might not be in $\mathscr{H}$. For all finite $F$,
we showed that $P_{F}(\delta_{x})\in\mathscr{H}$. But for $\delta_{x}$
be in $\mathscr{H}$, we must have the additional boundedness assumption
(\ref{eq:d3}) satisfied; see Theorem \ref{thm:del}.\end{rem}
\begin{lem}
\label{lem:proj2}Let $F\in\mathscr{F}\left(V\right)$, $x_{1}\in F$,
then 
\begin{equation}
\left(K_{F}^{-1}\delta_{x_{1}}\right)\left(x_{1}\right)=\left\Vert P_{F}\delta_{x_{1}}\right\Vert _{\mathscr{H}}^{2}.\label{eq:pd8}
\end{equation}
\end{lem}
\begin{proof}
Setting $\zeta^{\left(F\right)}:=K_{F}^{-1}\left(\delta_{x_{1}}\right)$,
we have 
\[
P_{F}\left(\delta_{x_{1}}\right)=\sum_{y\in F}\zeta^{\left(F\right)}\left(y\right)k_{F}\left(\cdot,y\right)
\]
and for all $z\in F$, 
\begin{eqnarray}
\underset{\zeta^{\left(F\right)}\left(x_{1}\right)}{\underbrace{\sum_{z\in F}\zeta^{\left(F\right)}\left(z\right)P_{F}\left(\delta_{x_{1}}\right)\left(z\right)}} & = & \sum_{F}\sum_{F}\zeta^{\left(F\right)}\left(z\right)\zeta^{\left(F\right)}\left(y\right)k_{F}\left(z,y\right)\label{eq:pd81}\\
 & = & \left\Vert P_{F}\delta_{x_{1}}\right\Vert _{\mathscr{H}}^{2}.\nonumber 
\end{eqnarray}
Note the LHS of (\ref{eq:pd81}) is given by (see Lemma \ref{lem:proj})
\begin{eqnarray*}
\left\Vert P_{F}\delta_{x_{1}}\right\Vert _{\mathscr{H}}^{2} & = & \left\langle P_{F}\delta_{x_{1}},\delta_{x_{1}}\right\rangle _{\mathscr{H}}\\
 & = & \sum_{y\in F}\left(K_{F}^{-1}\delta_{x_{1}}\right)\left(y\right)\left\langle k_{y},\delta_{x_{1}}\right\rangle _{\mathscr{H}}\\
 & = & \left(K_{F}^{-1}\delta_{x_{1}}\right)\left(x_{1}\right)=K_{F}^{-1}\left(x_{1},x_{1}\right).
\end{eqnarray*}
\end{proof}
\begin{cor}
\label{cor:proj1}If $\delta_{x_{1}}\in\mathscr{H}$ (see Theorem
\ref{thm:del}), then 
\begin{equation}
\sup_{F\in\mathscr{F}\left(V\right)}\left(K_{F}^{-1}\delta_{x_{1}}\right)\left(x_{1}\right)=\left\Vert \delta_{x_{1}}\right\Vert _{\mathscr{H}}^{2}.\label{eq:pd9}
\end{equation}

\end{cor}
The following condition is satisfied in some examples, but not all:
\begin{cor}
$\exists F\in\mathscr{F}\left(V\right)$ s.t. $\delta_{x_{1}}\in\mathscr{H}_{F}$
$\Longleftrightarrow$ 
\[
K_{F'}^{-1}\left(\delta_{x_{1}}\right)\left(x_{1}\right)=K_{F}^{-1}\left(\delta_{x_{1}}\right)\left(x_{1}\right)
\]
for all $F'\supset F$. 
\end{cor}

\begin{cor}[Monotonicity]
 If $F$ and $F'$ are in $\mathscr{F}\left(V\right)$ and $F\subset F'$,
then 
\begin{equation}
\left(K_{F}^{-1}\delta_{x_{1}}\right)\left(x_{1}\right)\leq\left(K_{F'}^{-1}\delta_{x_{1}}\right)\left(x_{1}\right)\label{eq:mono1}
\end{equation}
and 
\begin{equation}
\lim_{F\nearrow V}\left(K_{F}^{-1}\delta_{x_{1}}\right)\left(x_{1}\right)=\left\Vert \delta_{x_{1}}\right\Vert _{\mathscr{H}}^{2}.\label{eq:mono2}
\end{equation}
\end{cor}
\begin{proof}
By (\ref{eq:pd8}), 
\[
\left(K_{F}^{-1}\delta_{x_{1}}\right)\left(x_{1}\right)=\left\Vert P_{F}\delta_{x_{1}}\right\Vert _{\mathscr{H}}^{2}.
\]
Since $\mathscr{H}_{F}\subset\mathscr{H}_{F'}$, we have $P_{F}P_{F'}=P_{F}$,
so 
\[
\left\Vert P_{F}\delta_{x_{1}}\right\Vert _{\mathscr{H}}^{2}=\left\Vert P_{F}P_{F'}\delta_{x_{1}}\right\Vert _{\mathscr{H}}^{2}\leq\left\Vert P_{F'}\delta_{x_{1}}\right\Vert _{\mathscr{H}}^{2}
\]
i.e., 
\[
\left(K_{F}^{-1}\delta_{x_{1}}\right)\left(x_{1}\right)\leq\left(K_{F'}^{-1}\delta_{x_{1}}\right)\left(x_{1}\right).
\]
So (\ref{eq:mono1}) follows; and the limit in (\ref{eq:mono2}) is
monotone.\end{proof}
\begin{thm}
\label{thm:del}Given $V$, $k:V\times V\rightarrow\mathbb{R}$ positive
definite (p.d.). Let $\mathscr{H}=\mathscr{H}\left(k\right)$ be the
corresponding RKHS. Assume $V$ is countable and infinite. Then the
following three conditions \ref{enu:d1}-\ref{enu:d3} are equivalent;
$x_{1}\in V$ is fixed:
\begin{enumerate}[label=(\roman{enumi}),ref=(\roman{enumi})]
\item \label{enu:d1}$\delta_{x_{1}}\in\mathscr{H}$; 
\item \label{enu:d2}$\exists C_{x_{1}}<\infty$ such that for all $F\in\mathscr{F}\left(V\right)$,
the following estimate holds:
\begin{equation}
\left|\xi\left(x_{1}\right)\right|^{2}\leq C_{x_{1}}\underset{F\times F}{\sum\sum}\overline{\xi\left(x\right)}\xi\left(y\right)k\left(x,y\right)\label{eq:d1}
\end{equation}

\item \label{enu:d3}For $F\in\mathscr{F}\left(V\right)$, set 
\begin{equation}
K_{F}=\left(k\left(x,y\right)\right)_{\left(x,y\right)\in F\times F}\label{eq:d2}
\end{equation}
as a $\#F\times\#F$ matrix. See Definition \ref{def:d1}, eq. (\ref{eq:pd5}).
Then
\begin{equation}
\sup_{F\in\mathscr{F}\left(V\right)}\left(K_{F}^{-1}\delta_{x_{1}}\right)\left(x_{1}\right)<\infty.\label{eq:d3}
\end{equation}

\end{enumerate}
\end{thm}
\begin{proof}
\textbf{\ref{enu:d1}$\Rightarrow$\ref{enu:d2}} For $\xi\in l^{2}\left(F\right)$,
set 
\[
h_{\xi}=\sum_{y\in F}\xi\left(y\right)k_{y}\left(\cdot\right)\in\mathscr{H}_{F}.
\]
Then $\left\langle \delta_{x_{1}},h_{\xi}\right\rangle _{\mathscr{H}}=\xi\left(x_{1}\right)$
for all $\xi$. 

Since $\delta_{x_{1}}\in\mathscr{H}$, then by Schwarz: 
\begin{equation}
\left|\left\langle \delta_{x_{1}},h_{\xi}\right\rangle _{\mathscr{H}}\right|^{2}\leq\left\Vert \delta_{x_{1}}\right\Vert _{\mathscr{H}}^{2}\underset{F\times F}{\sum\sum}\overline{\xi\left(x\right)}\xi\left(y\right)k\left(x,y\right).\label{eq:d31}
\end{equation}
But $\left\langle \delta_{x_{1}},k_{y}\right\rangle _{\mathscr{H}}=\delta_{x_{1},y}=\begin{cases}
1 & y=x_{1}\\
0 & y\neq x_{1}
\end{cases}$; hence $\left\langle \delta_{x_{1}},h_{\xi}\right\rangle _{\mathscr{H}}=\xi\left(x_{1}\right)$,
and so (\ref{eq:d31}) implies (\ref{eq:d1}).

\textbf{\ref{enu:d2}$\Rightarrow$\ref{enu:d3}} Recall the matrix
\[
K_{F}:=\left(\left\langle k_{x},k_{y}\right\rangle \right)_{\left(x,y\right)\in F\times F}
\]
as a linear operator $l^{2}\left(F\right)\rightarrow l^{2}\left(F\right)$,
where 
\begin{equation}
\left(K_{F}\varphi\right)\left(x\right)=\sum_{y\in F}K_{F}\left(x,y\right)\varphi\left(y\right),\;\varphi\in l^{2}\left(F\right).
\end{equation}
By (\ref{eq:d1}), we have 
\begin{equation}
\ker\left(K_{F}\right)\subset\left\{ \varphi\in l^{2}\left(F\right):\varphi\left(x_{1}\right)=0\right\} .
\end{equation}
Equivalently, 
\begin{equation}
\ker\left(K_{F}\right)\subset\left\{ \delta_{x_{1}}\right\} ^{\perp}
\end{equation}
and so $\delta_{x_{1}}\Big|_{F}\in\ker\left(K_{F}\right)^{\perp}=\mbox{ran}\left(K_{F}\right)$,
and $\exists$ $\zeta^{\left(F\right)}\in l^{2}\left(F\right)$ s.t.
\begin{equation}
\delta_{x_{1}}\Big|_{F}=\underset{=:h_{F}}{\underbrace{\sum_{y\in F}\zeta^{\left(F\right)}\left(y\right)k\left(\cdot,y\right)}}.\label{eq:t0}
\end{equation}

\begin{claim*}
$P_{F}\left(\delta_{x_{1}}\right)=h_{F}$, where $P_{F}=$ projection
onto $\mathscr{H}_{F}$; see (\ref{eq:pd4}) and Lemma \ref{lem:proj1}.
(See Fig \ref{fig:proj}.)

\begin{figure}
\includegraphics[width=0.4\textwidth]{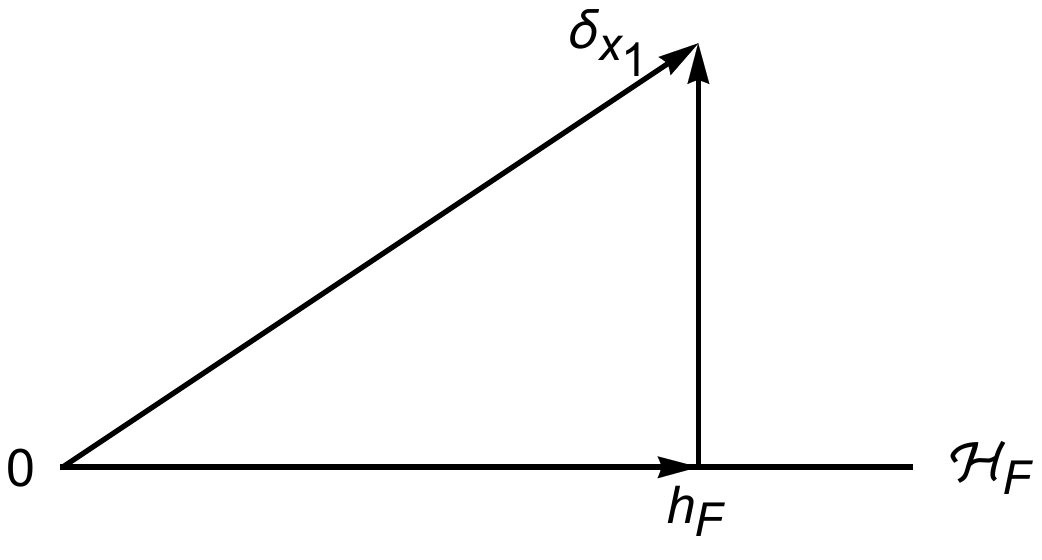}

\protect\caption{\label{fig:proj}$h_{F}:=P_{F}\left(\delta_{x_{1}}\right)$}

\end{figure}
\end{claim*}
\begin{proof}[Proof of the claim]
We only need to prove that $\delta_{x_{1}}-h_{F}\in\mathscr{H}\ominus\mathscr{H}_{F}$,
i.e., 
\begin{equation}
\left\langle \delta_{x_{1}}-h_{F},k_{z}\right\rangle _{\mathscr{H}}=0,\;\forall z\in F.\label{eq:t1}
\end{equation}
But, by (\ref{eq:t0}), 
\[
\text{LHS}_{\left(\ref{eq:t1}\right)}=\delta_{x_{1},z}-\sum_{y\in F}k\left(z,y\right)\zeta^{\left(F\right)}\left(y\right)=0.
\]

\end{proof}
If $F\subset F'$, $F,F'\in\mathscr{F}\left(V\right)$, then $\mathscr{H}_{F}\subset\mathscr{H}_{F'}$,
and $P_{F}P_{F'}=P_{F}$ by easy facts for projections. Hence 
\[
\left\Vert P_{F}\delta_{x_{1}}\right\Vert _{\mathscr{H}}^{2}\leq\left\Vert P_{F'}\delta_{x_{1}}\right\Vert _{\mathscr{H}}^{2},\quad h_{F}:=P_{F}\left(\delta_{x_{1}}\right)
\]
and 
\[
\lim_{F\nearrow V}\left\Vert \delta_{x_{1}}-h_{F}\right\Vert _{\mathscr{H}}=0.
\]

\textbf{\ref{enu:d3}$\Rightarrow$\ref{enu:d1}} Follows from Lemma
\ref{lem:proj2} and Corollary \ref{cor:proj1}.\end{proof}
\begin{cor}
The numbers $\left(\zeta^{\left(F\right)}\left(y\right)\right)_{y\in F}$
in (\ref{eq:t0}) satisfies
\begin{equation}
\zeta^{\left(F\right)}\left(x_{1}\right)=\underset{\left(y,z\right)\in F\times F}{\sum\sum}\zeta^{\left(F\right)}\left(y\right)\zeta^{\left(F\right)}\left(z\right)k\left(y,z\right).
\end{equation}
\end{cor}
\begin{proof}
Multiply (\ref{eq:t0}) by $\zeta^{\left(F\right)}\left(z\right)$
and carry out the summation. \end{proof}
\begin{rem}
To see that (\ref{eq:t0}) is a solution to a linear algebra problem,
with $F=\left\{ x_{i}\right\} _{i=1}^{n}$, note that (\ref{eq:t0})
$\Longleftrightarrow$ 
\begin{equation}
\begin{bmatrix}k\left(x_{1},x_{1}\right) & k\left(x_{1},x_{2}\right) & \cdots & k\left(x_{1},x_{n}\right)\\
k\left(x_{2},x_{1}\right) & k\left(x_{2},x_{2}\right) & \cdots & k\left(x_{2},x_{n}\right)\\
\vdots & \ddots & \ddots & \vdots\\
\vdots & \ddots & \ddots & \vdots\\
k\left(x_{n},x_{1}\right) & k\left(x_{n},x_{2}\right) & \cdots & k\left(x_{n},x_{n}\right)
\end{bmatrix}\begin{bmatrix}\zeta^{\left(F\right)}\left(x_{1}\right)\\
\zeta^{\left(F\right)}\left(x_{2}\right)\\
\vdots\\
\zeta^{\left(F\right)}\left(x_{n-1}\right)\\
\zeta^{\left(F\right)}\left(x_{n}\right)
\end{bmatrix}=\begin{bmatrix}1\\
0\\
\vdots\\
0\\
0
\end{bmatrix}\label{eq:p1}
\end{equation}

\end{rem}
We now resume the general case of $k$ given and positive definite
on $V\times V$.
\begin{cor}
We have
\begin{equation}
\zeta^{\left(F\right)}\left(x_{1}\right)=\left\Vert P_{F}\left(\delta_{x_{1}}\right)\right\Vert _{\mathscr{H}}^{2}
\end{equation}
where 
\begin{equation}
P_{F}\left(\delta_{x_{1}}\right)=\sum_{y\in F}\zeta^{\left(F\right)}\left(y\right)k_{y}\left(\cdot\right)
\end{equation}
and 
\begin{equation}
\zeta^{\left(F\right)}=K_{N}^{-1}\left(\delta_{x_{1}}\right),\quad N:=\#F.
\end{equation}
\end{cor}
\begin{proof}
It follows from (\ref{eq:p1}) that
\[
\sum_{j}k\left(x_{i},x_{j}\right)\zeta^{\left(F\right)}\left(x_{j}\right)=\delta_{1,i}
\]
and so multiplying by $\zeta^{\left(F\right)}\left(i\right)$, and
summing over $i$, gives 
\[
\underset{=\left\Vert P_{F}\left(\delta_{x_{1}}\right)\right\Vert _{\mathscr{H}}^{2}}{\underbrace{\sum_{i}\sum_{j}k\left(x_{i},x_{j}\right)\zeta^{\left(F\right)}\left(x_{i}\right)\zeta^{\left(F\right)}\left(x_{j}\right)}}=\zeta^{\left(F\right)}\left(x_{1}\right).
\]
\end{proof}
\begin{cor}
We have
\begin{enumerate}
\item[(i)] 
\begin{equation}
P_{F}\left(\delta_{x_{1}}\right)=\zeta^{\left(F\right)}\left(x_{1}\right)k_{x_{1}}+\sum_{y\in F\backslash\left\{ x_{1}\right\} }\zeta^{\left(F\right)}\left(y\right)k_{y}
\end{equation}
 where $\zeta_{F}$ solves (\ref{eq:p1}), for all $F\in\mathscr{F}\left(V\right)$; 
\item[(ii)]  
\begin{equation}
\left\Vert P_{F}\left(\delta_{x_{1}}\right)\right\Vert _{\mathscr{H}}^{2}=\zeta^{\left(F\right)}\left(x_{1}\right)\label{eq:p4}
\end{equation}
and so in particular:
\item[(iii)] 
\begin{equation}
0<\zeta^{\left(F\right)}\left(x_{1}\right)\leq\left\Vert \delta_{x_{1}}\right\Vert _{\mathscr{H}}^{2}\label{eq:p3}
\end{equation}

\end{enumerate}
\end{cor}
\begin{proof}
Formula (\ref{eq:p4}) follows from the definition of $\zeta^{\left(F\right)}$
as a solution to the matrix problem $K_{N}\zeta^{\left(F\right)}=\delta_{x_{1}}$,
but we may also prove (\ref{eq:p4}) directly from 
\begin{equation}
P_{F}\left(\delta_{x_{1}}\right)=\sum_{y}\zeta^{\left(F\right)}\left(y\right)k_{y}\,.\label{eq:p5}
\end{equation}
Apply $\left\langle \cdot,\delta_{x_{1}}\right\rangle _{\mathscr{H}}$
to both sides in (\ref{eq:p5}), we get 
\[
\underset{\left\Vert P_{F}\left(\delta_{x_{1}}\right)\right\Vert _{\mathscr{H}}^{2}}{\underbrace{\left\langle \delta_{x_{1}},P_{F}\left(\delta_{x_{1}}\right)\right\rangle _{\mathscr{H}}}}=\zeta^{\left(F\right)}\left(x_{1}\right)
\]
since $P_{F}=P_{F}^{*}=P_{F}^{2}$; i.e., a projection in the RKHS
$\mathscr{H}=\mathscr{H}_{V}$ of $k$. \end{proof}
\begin{example}[$\#F=2$]
 Let $F=\left\{ x_{1},x_{2}\right\} $, $K_{F}=\left(k_{ij}\right)_{i,j=1}^{2}$,
where $k_{ij}:=k\left(x_{i},x_{j}\right)$. Then (\ref{eq:p1}) reads
\begin{equation}
\begin{bmatrix}k_{11} & k_{12}\\
k_{21} & k_{22}
\end{bmatrix}\begin{bmatrix}\zeta_{F}\left(x_{1}\right)\\
\zeta_{F}\left(x_{2}\right)
\end{bmatrix}=\begin{bmatrix}1\\
0
\end{bmatrix}.
\end{equation}
Set $D:=\det\left(K_{F}\right)=k_{11}k_{22}-k_{12}k_{21}$, then:
\[
\zeta_{F}\left(x_{1}\right)=\frac{k_{22}}{D},\quad\zeta_{F}\left(x_{2}\right)=-\frac{k_{21}}{D}.
\]

\end{example}

\begin{example}
Let $V=\left\{ x_{1},x_{2},\ldots\right\} $ be an ordered set. Set
$F_{n}:=\left\{ x_{1},\ldots,x_{n}\right\} $. Note that with 
\begin{equation}
D_{n}=\det\left(K_{F_{n}}\right)=\det\left(\left(k\left(x_{i},x_{j}\right)\right)_{i,j=1}^{n}\right),\;\mbox{and}
\end{equation}
\begin{equation}
D'_{n-1}=\left(1,1\right)\;\text{minor of }K_{F_{n}}=\det\left(\left(k\left(x_{i},x_{j}\right)\right)_{i,j=2}^{n}\right);
\end{equation}
then 
\begin{equation}
\zeta^{\left(F_{n}\right)}\left(x_{1}\right)=\frac{D'_{n-1}}{D_{n}}=\left(K_{F_{n}}^{-1}\delta_{x_{1}}\right)\left(x_{1}\right).\label{eq:p2}
\end{equation}
 \end{example}
\begin{cor}
We have 
\[
\frac{1}{k\left(x_{1},x_{1}\right)}\leq\frac{k\left(x_{2},x_{2}\right)}{D_{2}}\leq\cdots\leq\frac{D'_{n-1}}{D_{n}}\leq\cdots\leq\left\Vert \delta_{x_{1}}\right\Vert _{\mathscr{H}}^{2}.
\]
\end{cor}
\begin{proof}
Follows from (\ref{eq:p2}), and if $F\subset F'$ are two finite
subsets, then 
\[
\left\Vert P_{F}\left(\delta_{x_{1}}\right)\right\Vert _{\mathscr{H}}^{2}\leq\left\Vert P_{F'}\left(\delta_{x_{1}}\right)\right\Vert _{\mathscr{H}}^{2}\leq\left\Vert \delta_{x_{1}}\right\Vert _{\mathscr{H}}^{2}.
\]

\end{proof}
Let $k:V\times V\rightarrow\mathbb{R}$ be as specified above. Let
$\mathscr{H}=\mathscr{H}\left(k\right)$ be the RKHS. We set $\mathscr{F}\left(V\right)$:=
all finite subsets of $V$; and if $x\in V$ is fixed, $\mathscr{F}_{x}\left(V\right):=\left\{ F\in\mathscr{F}\left(V\right)\:|\: x\in F\right\} $. 

For $F\in\mathscr{F}\left(V\right)$, let $K_{F}$ be the $\#F\times\#F$
matrix given by $\left(k\left(x,y\right)\right)_{\left(x,y\right)\in F\times F}$.
Following \cite{KZ96}, we say that $k$ is \emph{strictly positive}
iff (Def.) $\det K_{F}>0$ for all $F\in\mathscr{F}\left(V\right)$. 

Set $D_{F}:=\det K_{F}$. If $x\in V$, and $F\in\mathscr{F}_{x}\left(V\right)$,
set $K'_{F}:=$ the minor in $K_{F}$ obtained by omitting row $x$
and column $x$, see Fig \ref{fig:minor}.

\begin{figure}[H]
\begin{tabular}[t]{ccc}
\includegraphics[width=0.3\columnwidth]{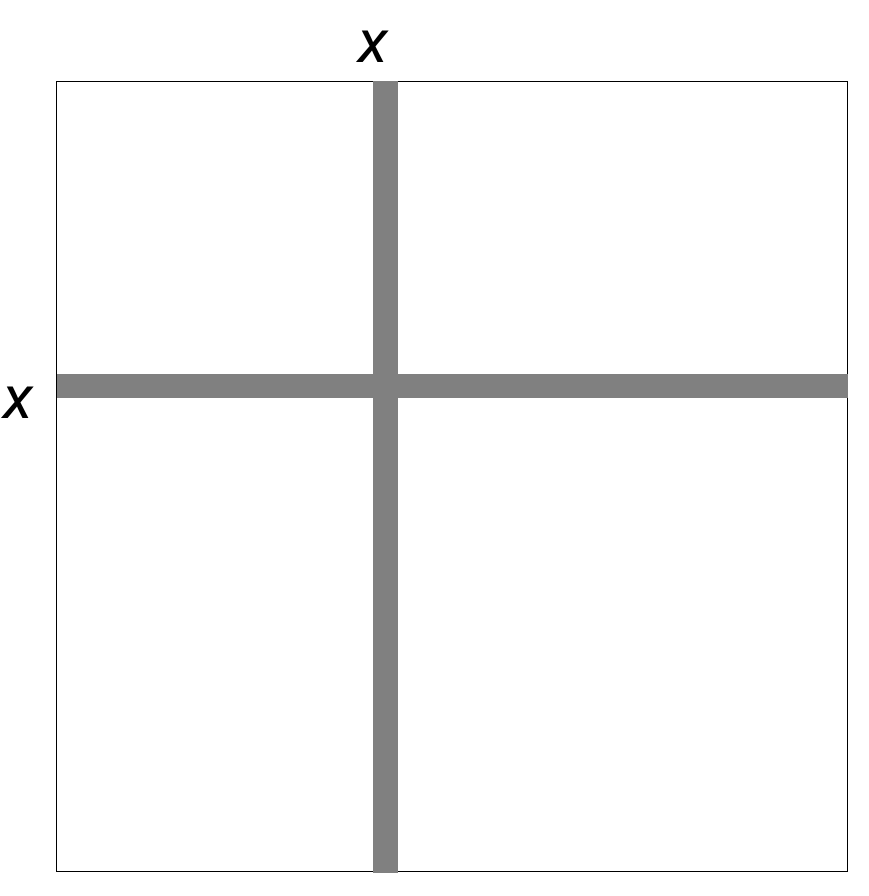} &  & \includegraphics[width=0.3\columnwidth]{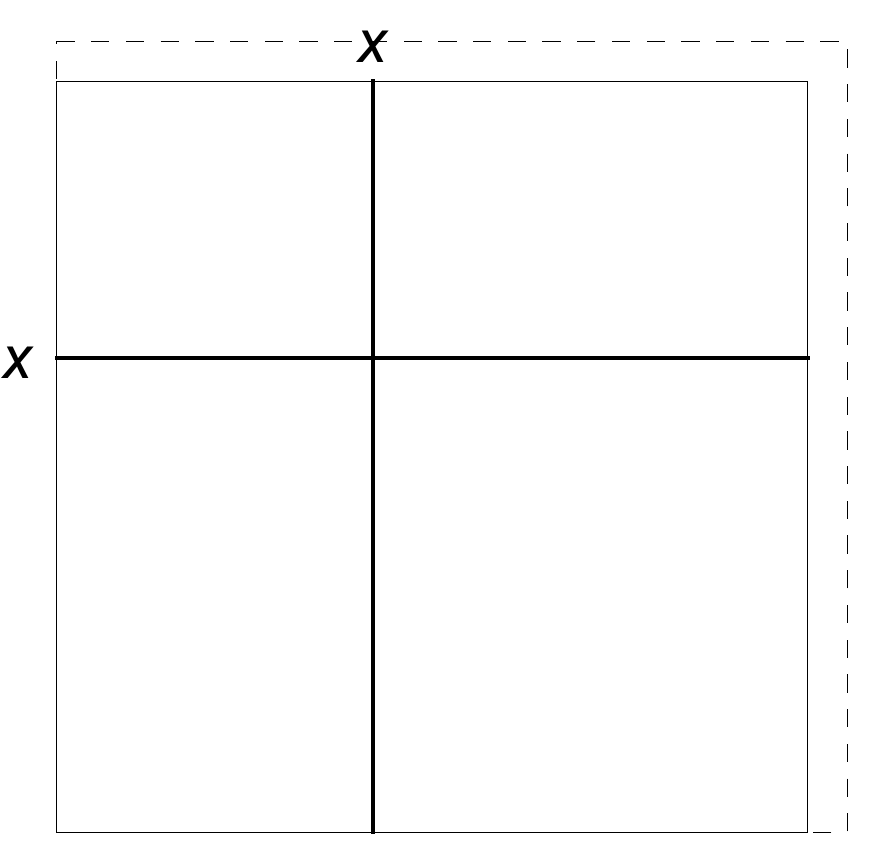}\tabularnewline
\end{tabular}

\protect\caption{\label{fig:minor}The $\left(x,x\right)$ minors, $K_{F}\rightarrow K_{F}'$. }

\end{figure}

\begin{cor}
\label{cor:sp}Suppose $k:V\times V\rightarrow\mathbb{R}$ is strictly
positive. Let $x\in V$. Then 
\begin{equation}
\delta_{x}\in\mathscr{H}\Longleftrightarrow\sup_{F\in\mathscr{F}_{x}\left(V\right)}\frac{D'_{F}}{D_{F}}<\infty.
\end{equation}

\end{cor}

\subsection{Unbounded containment in RKHSs}
\begin{defn}
Let $\mathscr{K}$ and $\mathscr{H}$ be two Hilbert spaces. We say
that $\mathscr{K}$ is \emph{unboundedly contained} in $\mathscr{H}$
if there is a dense subspace $\mathscr{K}_{0}\subset\mathscr{K}$
such that $\mathscr{K}_{0}\subset\mathscr{H}$; and the inclusion
operator, with $\mathscr{K}_{0}$ as its dense domain, is \emph{closed},
i.e., 
\[
\mathscr{K}\overset{incl}{\hookrightarrow}\mathscr{H},\quad dom\left(incl\right)=\mathscr{K}_{0}.
\]
Let $k:V\times V\rightarrow\mathbb{R}$ be a p.d. kernel, and let
$\mathscr{H}$ be the corresponding RKHS. Set $\mathscr{K}=l^{2}\left(V\right)$,
and 
\begin{equation}
\mathscr{K}_{0}=span\left\{ \delta_{x}\:|\: x\in V\right\} .\label{eq:in1}
\end{equation}
\end{defn}
\begin{prop}
\label{prop:ubc}If $\delta_{x}\in\mathscr{H}$ for $\forall x\in V$,
then $l^{2}\left(V\right)$ is unboundedly contained in $\mathscr{H}$. \end{prop}
\begin{proof}
Recall that $\mathscr{H}$ is the RKHS defined for a fixed p.d. kernel
$k:V\times V\rightarrow\mathbb{R}$. Let $k_{x}$ be the vector in
$\mathscr{H}$, given by $k_{x}\left(y\right)=k\left(x,y\right)$,
s.t. 
\begin{equation}
f\left(x\right)=\left\langle k_{x},f\right\rangle _{\mathscr{H}},\quad\forall f\in\mathscr{H}.\label{eq:in2}
\end{equation}
To finish the proof we will need:\end{proof}
\begin{lem}
\label{lem:in1}The following equation 
\begin{equation}
\left\langle \delta_{x},k_{y}\right\rangle _{\mathscr{H}}=\delta_{x,y}\label{eq:in3}
\end{equation}
holds if $\delta_{x}\in\mathscr{H}$ for $\forall x\in V$. \end{lem}
\begin{proof}
(\ref{eq:in3}) is immediate from (\ref{eq:in2}). \end{proof}
\begin{lem}
On 
\begin{equation}
span\left\{ k_{x}\:|\: x\in V\right\} \subset\mathscr{H}\label{eq:in4}
\end{equation}
define $Mk_{x}:=\delta_{x}$, then by Lemma \ref{lem:in1}, $M$ extends
to be a well defined operator $M:\mathscr{H}\rightarrow l^{2}\left(V\right)$
with dense domain (\ref{eq:in4}). We have 
\begin{equation}
\left\langle k,Mf\right\rangle _{l^{2}\left(V\right)}=\left\langle k,f\right\rangle _{\mathscr{H}},\quad\forall k\in span\left\{ \delta_{x}\right\} ,\;\forall f\in dom\left(M\right).\label{eq:in5}
\end{equation}
\end{lem}
\begin{proof}
By linearity, it is enough to prove that
\begin{equation}
\left\langle \delta_{x},\delta_{y}\right\rangle _{l^{2}}=\left\langle \delta_{x},k_{y}\right\rangle _{\mathscr{H}}\label{eq:in6}
\end{equation}
holds for $\forall x,y\in V$. But (\ref{eq:in6}) follows immediate
from Lemma \ref{lem:in1}.\end{proof}
\begin{cor}
If $L:l^{2}\left(V\right)\rightarrow\mathscr{H}$ denotes the inclusion
mapping with $dom\left(L\right)=span\left\{ \delta_{x}:x\in V\right\} $,
then we conclude that 
\begin{equation}
L\subset M^{*},\;\mbox{and}\; M\subset L^{*}.
\end{equation}
Since $dom\left(M\right)$ is dense in $\mathscr{H}$, it follows
that $L^{*}$ has dense domain; and that therefore $L$ is closable. \end{cor}
\begin{rem}
This also completes the proof of Proposition \ref{prop:ubc}.\end{rem}
\begin{cor}
Suppose $k:V\times V\rightarrow\mathbb{R}$ is as given, and that
$\mathscr{H}=RKHS\left(k\right)$. Let $L$ be the densely defined
inclusion mapping $l^{2}\left(V\right)\rightarrow\mathscr{H}$. Then
$L^{*}L$ is selfadjoint with dense domain in $l^{2}\left(V\right)$;
and $LL^{*}$ is selfadjoint with dense domain in $\mathscr{H}$.
Moreover, the following polar decomposition holds:
\begin{equation}
L=U\left(L^{*}L\right)^{1/2}=\left(LL^{*}\right)^{1/2}U
\end{equation}
where $U$ is a partial isometry $l^{2}\left(V\right)\rightarrow\mathscr{H}$. 
\end{cor}

\section{\label{sec:egs}Point-masses in concrete models}

Suppose $V\subset D\subset\mathbb{R}^{d}$ where $V$ is countable
and discrete, but $D$ is open. In this case, we get two kernels:
$k$ on $D\times D$, and $k_{V}:=k\big|_{V\times V}$ on $V\times V$
by restriction. If $x\in V$, then $k_{x}^{\left(V\right)}\left(\cdot\right)=k\left(\cdot,x\right)$
is a function on $V$, while $k_{x}\left(\cdot\right)=k\left(\cdot,x\right)$
is a function on $D$. 

This means that the corresponding RKHSs are different, $\mathscr{H}_{V}$
vs $\mathscr{H}$, where $\mathscr{H}_{V}=$ a RKHS of functions on
$V$, and $\mathscr{H}=$ a RKHS of functions on $D$. 
\begin{lem}
\label{lem:mc1}$\mathscr{H}_{V}$ is isometrically contained in $\mathscr{H}$
via $k_{x}^{\left(V\right)}\longmapsto k_{x}$, $x\in V$. \end{lem}
\begin{proof}
If $F\subset V$ is a finite subset, and $\xi=\xi_{F}$ is a function
on $F$, then 
\[
\left\Vert \sum\nolimits _{x\in F}\xi\left(x\right)k_{x}^{\left(V\right)}\right\Vert _{\mathscr{H}_{V}}=\left\Vert \sum\nolimits _{x\in F}\xi\left(x\right)k_{x}\right\Vert _{\mathscr{H}}.
\]
The desired result follows from this. 
\end{proof}
We are concerned with cases of kernels $k:D\times D\rightarrow\mathbb{R}$
with restriction $k_{V}:V\times V\rightarrow\mathbb{R}$, where $V$
is a countable discrete subset of $D$. Typically, for $x\in V$,
we may have (restriction) $\delta_{x}\big|_{V}\in\mathscr{H}_{V}$,
but $\delta_{x}\notin\mathscr{H}$; indeed this happens for the kernel
$k$ of standard Brownian motion: 

$D=\mathbb{R}_{+}$;

$V=$ an ordered subset $0<x_{1}<x_{2}<\cdots<x_{i}<x_{i+1}<\cdots$,
$V=\left\{ x_{i}\right\} _{i=1}^{\infty}$. 

In this case, we compute $\mathscr{H}_{V}$, and we show that $\delta_{x_{i}}\big|_{V}\in\mathscr{H}_{V}$;
while for $\mathscr{H}_{m}=$ the Cameron-Martin Hilbert space, we
have $\delta_{x_{i}}\notin\mathscr{H}_{m}$. 

Also note that $\delta_{x_{1}}$ has a different meaning with reference
to $\mathscr{H}_{V}$ vs $\mathscr{H}_{m}$. In the first case, it
is simply $\delta_{x_{1}}\left(y\right)=\begin{cases}
1 & y=x_{1}\\
0 & y\in V\backslash\left\{ x_{1}\right\} 
\end{cases}$. In the second case, $\delta_{x_{1}}$ is a Schwartz distribution.
We shall abuse notation, writing $\delta_{x}$ in both cases. 

In the following, we will consider restriction to $V\times V$ of
a special continuous p.d. kernel $k$ on $\mathbb{R}_{+}\times\mathbb{R}_{+}$.
It is $k\left(s,t\right)=s\wedge t=\min\left(s,t\right)$. Before
we restrict, note that the RKHS of this $k$ is the Cameron-Martin
Hilbert space of function $f$ on $\mathbb{R}_{+}$ with distribution
derivative $f'\in L^{2}\left(\mathbb{R}_{+}\right)$, and 
\begin{equation}
\left\Vert f\right\Vert _{\mathscr{H}}^{2}:=\int_{0}^{\infty}\left|f'\left(t\right)\right|^{2}dt<\infty.\label{eq:cm1}
\end{equation}
For details, see below. 

\textbf{Application.} The Hilbert space given by $\left\Vert \cdot\right\Vert _{\mathscr{H}}^{2}$
in (\ref{eq:cm1}) is called the Cameron-Martin Hilbert space, and,
as noted, it is the RKHS of $k:\mathbb{R}_{+}\times\mathbb{R}_{+}\rightarrow\mathbb{R}:$
$k\left(s,t\right):=s\wedge t$. Now pick a discrete subset $V\subset\mathbb{R}_{+}$;
then Lemma \ref{lem:mc1} states that the RKHS of the $V\times V$
restricted kernel, $k^{\left(V\right)}$ is isometrically embedded
into $\mathscr{H}$, i.e., setting 
\begin{equation}
J^{\left(V\right)}\left(k_{x}^{\left(V\right)}\right)=k_{x},\quad\forall x\in V;\label{eq:cm2}
\end{equation}
$J^{\left(V\right)}$ extends by ``closed span'' to an isometry
$\mathscr{H}_{V}\xrightarrow{J^{\left(V\right)}}\mathscr{H}$. It
further follows from the lemma, that the range of $J^{\left(V\right)}$
may have infinite co-dimension.

Note that $P_{V}:=J^{\left(V\right)}\left(J^{\left(V\right)}\right)^{*}$
is the projection onto the range of $J^{\left(V\right)}$. The ortho-complement
is as follow: 
\begin{equation}
\mathscr{H}\ominus\mathscr{H}_{V}=\left\{ \psi\in\mathscr{H}\:\big|\:\psi\left(x\right)=0,\;\forall x\in V\right\} .\label{eq:cm3}
\end{equation}

\begin{example}
Let $k$ and $k^{\left(V\right)}$ be as in (\ref{eq:cm2}), and set
$V:=\pi\mathbb{Z}_{+}$, i.e., integer multiples of $\pi$. Then easy
generators of wavelet functions (see e.g., \cite{BJ02}) yield non-zero
functions $\psi$ on $\mathbb{R}_{+}$ such that 
\begin{equation}
\psi\in\mathscr{H}\ominus\mathscr{H}_{V}.\label{eq:cm4}
\end{equation}
More precisely, 
\begin{equation}
0<\int_{0}^{\infty}\left|\psi'\left(t\right)\right|^{2}dt<\infty,\label{eq:cm5}
\end{equation}
where $\psi'$ is the distribution (weak) derivative; and 
\begin{equation}
\psi\left(n\pi\right)=0,\quad\forall n\in\mathbb{Z}_{+}.\label{eq:cm6}
\end{equation}
An explicit solution to (\ref{eq:cm4})-(\ref{eq:cm6}) is 
\begin{equation}
\psi\left(t\right)=\prod_{n=1}^{\infty}\cos\left(\frac{t}{2^{n}}\right)=\frac{\sin t}{t},\quad\forall t\in\mathbb{R}.\label{eq:cm7}
\end{equation}
From this, one easily generates an infinite-dimensional set of solutions. 
\end{example}

\subsection{\label{sub:bm}Brownian motion}

Consider the covariance function of standard Brownian motion $B_{t}$,
$t\in[0,\infty)$, i.e., a Gaussian process $\left\{ B_{t}\right\} $
with mean zero and covariance function 
\begin{equation}
\mathbb{E}\left(B_{s}B_{t}\right)=s\wedge t=\min\left(s,t\right).\label{eq:bm1}
\end{equation}
We now show that the restriction of (\ref{eq:bm1}) to $V\times V$
for an ordered subset (we fix such a set $V$):
\begin{equation}
V:\;0<x_{1}<x_{2}<\cdots<x_{i}<x_{i+1}<\cdots\label{eq:bm2}
\end{equation}
has the discrete mass property (Def. \ref{def:dmp}). 

Set $\mathscr{H}_{V}=RKHS(k\big|_{V\times V})$, 
\begin{equation}
k_{V}\left(x_{i},x_{j}\right)=x_{i}\wedge x_{j}.\label{eq:bm3}
\end{equation}
We consider the set $F_{n}=\left\{ x_{1},x_{2},\ldots,x_{n}\right\} $
of finite subsets of $V$, and 
\begin{equation}
K_{n}=k^{\left(F_{n}\right)}=\begin{bmatrix}x_{1} & x_{1} & x_{1} & \cdots & x_{1}\\
x_{1} & x_{2} & x_{2} & \cdots & x_{2}\\
x_{1} & x_{2} & x_{3} & \cdots & x_{3}\\
\vdots & \vdots & \vdots & \vdots & \vdots\\
x_{1} & x_{2} & x_{3} & \cdots & x_{n}
\end{bmatrix}=\left(x_{i}\wedge x_{j}\right)_{i,j=1}^{n}.\label{eq:bm4}
\end{equation}
We will show that condition \ref{enu:d3} in Theorem \ref{thm:del}
holds for $k_{V}$. For this, we must compute all the determinants,
$D_{n}=\det\left(K_{F}\right)$ etc. ($n=\#F$), see Corollary \ref{cor:sp}.
\begin{lem}
~ 
\begin{equation}
D_{n}=\det\left(\left(x_{i}\wedge x_{j}\right)_{i,j=1}^{n}\right)=x_{1}\left(x_{2}-x_{1}\right)\left(x_{3}-x_{2}\right)\cdots\left(x_{n}-x_{n-1}\right).\label{eq:bm5}
\end{equation}
\end{lem}
\begin{proof}
Induction. In fact, 
\[
\begin{bmatrix}x_{1} & x_{1} & x_{1} & \cdots & x_{1}\\
x_{1} & x_{2} & x_{2} & \cdots & x_{2}\\
x_{1} & x_{2} & x_{3} & \cdots & x_{3}\\
\vdots & \vdots & \vdots & \vdots & \vdots\\
x_{1} & x_{2} & x_{3} & \cdots & x_{n}
\end{bmatrix}\sim\begin{bmatrix}x_{1} & 0 & 0 & \cdots & 0\\
0 & x_{2}-x_{1} & 0 & \cdots & 0\\
0 & 0 & x_{3}-x_{2} & \cdots & 0\\
\vdots & \vdots & \vdots & \ddots & \vdots\\
0 & \cdots & 0 & \cdots & x_{n}-x_{n-1}
\end{bmatrix},
\]
unitary equivalence in finite dimensions.
\end{proof}

\begin{lem}
Let 
\begin{equation}
\zeta_{\left(n\right)}:=K_{n}^{-1}\left(\delta_{x_{1}}\right)\left(\cdot\right)\label{eq:bm7}
\end{equation}
be as in eq. (\ref{eq:pd8}), so that 
\begin{equation}
\left\Vert P_{F_{n}}\left(\delta_{x_{1}}\right)\right\Vert _{\mathscr{H}_{V}}^{2}=\zeta_{\left(n\right)}\left(x_{1}\right).\label{eq:bm8}
\end{equation}
Then, 
\begin{eqnarray*}
\zeta_{\left(1\right)}\left(x_{1}\right) & = & \frac{1}{x_{1}}\\
\zeta_{\left(n\right)}\left(x_{1}\right) & = & \frac{x_{2}}{x_{1}\left(x_{2}-x_{1}\right)},\quad\text{for}\; n=2,3,\ldots,
\end{eqnarray*}
and 
\[
\left\Vert \delta_{x_{1}}\right\Vert _{\mathscr{H}_{V}}^{2}=\frac{x_{2}}{x_{1}\left(x_{2}-x_{1}\right)}.
\]
\end{lem}
\begin{proof}
A direct computation shows the $\left(1,1\right)$ minor of the matrix
$K_{n}^{-1}$ is
\begin{equation}
D'_{n-1}=\det\left(\left(x_{i}\wedge x_{j}\right)_{i,j=2}^{n}\right)=x_{2}\left(x_{3}-x_{2}\right)\left(x_{4}-x_{3}\right)\cdots\left(x_{n}-x_{n-1}\right)\label{eq:bm6}
\end{equation}
and so 
\begin{eqnarray*}
\zeta_{\left(1\right)}\left(x_{1}\right) & = & \frac{1}{x_{1}},\quad\mbox{and}\\
\zeta_{\left(2\right)}\left(x_{1}\right) & = & \frac{x_{2}}{x_{1}\left(x_{2}-x_{1}\right)}\\
\zeta_{\left(3\right)}\left(x_{1}\right) & = & \frac{x_{2}\left(x_{3}-x_{2}\right)}{x_{1}\left(x_{2}-x_{1}\right)\left(x_{3}-x_{2}\right)}=\frac{x_{2}}{x_{1}\left(x_{2}-x_{1}\right)}\\
\zeta_{\left(4\right)}\left(x_{1}\right) & = & \frac{x_{2}\left(x_{3}-x_{2}\right)\left(x_{4}-x_{3}\right)}{x_{1}\left(x_{2}-x_{1}\right)\left(x_{3}-x_{2}\right)\left(x_{4}-x_{3}\right)}=\frac{x_{2}}{x_{1}\left(x_{2}-x_{1}\right)}\\
 & \vdots
\end{eqnarray*}
The result follows from this, and from Corollary \ref{cor:proj1}.\end{proof}
\begin{cor}
\label{cor:proj}$P_{F_{n}}\left(\delta_{x_{1}}\right)=P_{F_{2}}\left(\delta_{x_{1}}\right)$,
$\forall n\geq2$. Therefore, 
\begin{equation}
\delta_{x_{1}}\in\mathscr{H}_{V}^{\left(F_{2}\right)}:=span\{k_{x_{1}}^{\left(V\right)},k_{x_{2}}^{\left(V\right)}\}
\end{equation}
and
\begin{equation}
\delta_{x_{1}}=\zeta_{\left(2\right)}\left(x_{1}\right)k_{x_{1}}^{\left(V\right)}+\zeta_{\left(2\right)}\left(x_{2}\right)k_{x_{2}}^{\left(V\right)}
\end{equation}
where 
\[
\zeta_{\left(2\right)}\left(x_{i}\right)=K_{2}^{-1}\left(\delta_{x_{1}}\right)\left(x_{i}\right),\; i=1,2.
\]
Specifically, 
\begin{eqnarray}
\zeta_{\left(2\right)}\left(x_{1}\right) & = & \frac{x_{2}}{x_{1}\left(x_{2}-x_{1}\right)}\\
\zeta_{\left(2\right)}\left(x_{2}\right) & = & \frac{-1}{x_{2}-x_{1}};
\end{eqnarray}
and 
\begin{equation}
\left\Vert \delta_{x_{1}}\right\Vert _{\mathscr{H}_{V}}^{2}=\frac{x_{2}}{x_{1}\left(x_{2}-x_{1}\right)}.\label{eq:dn}
\end{equation}
\end{cor}
\begin{proof}
Follows from the lemma. Note that 
\[
\zeta_{n}\left(x_{1}\right)=\left\Vert P_{F_{n}}\left(\delta_{x_{1}}\right)\right\Vert _{\mathscr{H}}^{2}
\]
and $\zeta_{\left(1\right)}\left(x_{1}\right)\leq\zeta_{\left(2\right)}\left(x_{1}\right)\leq\cdots$,
since $F_{n}=\left\{ x_{1},x_{2},\ldots,x_{n}\right\} $. In particular,
$\frac{1}{x_{1}}\leq\frac{x_{2}}{x_{1}\left(x_{2}-x_{1}\right)}$,
which yields (\ref{eq:dn}). \end{proof}
\begin{rem}
We showed that $\delta_{x_{1}}\in\mathscr{H}_{V}$, $V=\left\{ x_{1}<x_{2}<\cdots\right\} \subset\mathbb{R}_{+}$,
with the restriction of $s\wedge t$ = the covariance kernel of Brownian
motion. 

The same argument also shows that $\delta_{x_{i}}\in\mathscr{H}_{V}$
when $i>1$. We only need to modify the index notation from the case
of the proof for $\delta_{x_{1}}\in\mathscr{H}_{V}$. The details
are sketched below.

Fix $V=\left\{ x_{i}\right\} _{i=1}^{\infty}$, $x_{1}<x_{2}<\cdots$,
then 
\[
P_{F_{n}}\left(\delta_{x_{i}}\right)=\begin{cases}
0 & \text{if \ensuremath{n<i-1}}\\
\sum_{s=1}^{n}\left(K_{F_{n}}^{-1}\delta_{x_{i}}\right)\left(x_{s}\right)k_{x_{s}} & \text{if \ensuremath{n\geq i}}
\end{cases}
\]
and 
\[
\left\Vert P_{F_{n}}\left(\delta_{x_{i}}\right)\right\Vert _{\mathscr{H}}^{2}=\begin{cases}
0 & \text{if \ensuremath{n<i-1}}\\
\frac{1}{x_{i}-x_{i-1}} & \text{if \ensuremath{n=i}}\\
\frac{x_{i+1}-x_{i-1}}{\left(x_{i}-x_{i-1}\right)\left(x_{i+1}-x_{i}\right)} & \text{if \ensuremath{n>i}}
\end{cases}
\]
\textbf{Conclusion.} 
\begin{eqnarray}
\delta_{x_{i}} & \in & span\left\{ k_{x_{i-1}}^{\left(V\right)},k_{x_{i}}^{\left(V\right)},k_{x_{i+1}}^{\left(V\right)}\right\} ,\quad\mbox{and}\\
\left\Vert \delta_{x_{i}}\right\Vert _{\mathscr{H}}^{2} & = & \frac{x_{i+1}-x_{i-1}}{\left(x_{i}-x_{i-1}\right)\left(x_{i+1}-x_{i}\right)}.
\end{eqnarray}
\end{rem}
\begin{cor}
Let $V\subset\mathbb{R}_{+}$ be countable. If $x_{a}\in V$ is an
accumulation point (from $V$), then $\left\Vert \delta_{a}\right\Vert _{\mathscr{H}_{V}}=\infty$. \end{cor}
\begin{rem}
This computation will be revisited in sect. \ref{sec:net}, in a much
wider context. \end{rem}
\begin{example}
An illustration for $0<x_{1}<x_{2}<x_{3}<x_{4}$: 
\begin{eqnarray*}
P_{F}\left(\delta_{x_{3}}\right) & = & \sum_{y\in F}\zeta^{\left(F\right)}\left(y\right)k_{y}\left(\cdot\right)\\
\zeta^{\left(F\right)} & = & K_{F}^{-1}\delta_{x_{3}}\;.
\end{eqnarray*}
That is, 
\[
\underset{\left(K_{F}\left(x_{i},x_{j}\right)\right)_{i,j=1}^{4}}{\underbrace{\begin{bmatrix}x_{1} & x_{1} & x_{1} & x_{1}\\
x_{1} & x_{2} & x_{2} & x_{2}\\
x_{1} & x_{2} & x_{3} & x_{3}\\
x_{1} & x_{2} & x_{3} & x_{4}
\end{bmatrix}}}\begin{bmatrix}\zeta^{\left(F\right)}\left(x_{1}\right)\\
\zeta^{\left(F\right)}\left(x_{2}\right)\\
\zeta^{\left(F\right)}\left(x_{3}\right)\\
\zeta^{\left(F\right)}\left(x_{4}\right)
\end{bmatrix}=\begin{bmatrix}0\\
0\\
1\\
0
\end{bmatrix}
\]
and 
\begin{eqnarray*}
\zeta^{\left(F\right)}\left(x_{3}\right) & = & \frac{x_{1}\left(x_{2}-x_{1}\right)\left(x_{4}-x_{2}\right)}{x_{1}\left(x_{2}-x_{1}\right)\left(x_{3}-x_{2}\right)\left(x_{4}-x_{3}\right)}\\
 & = & \frac{x_{4}-x_{2}}{\left(x_{3}-x_{2}\right)\left(x_{4}-x_{3}\right)}=\left\Vert \delta_{x_{3}}\right\Vert _{\mathscr{H}}^{2}.
\end{eqnarray*}

\end{example}

\begin{example}[Sparse sample-points]
Let $V=\left\{ x_{i}\right\} _{i=1}^{\infty}$, where 
\[
x_{i}=\frac{i\left(i-1\right)}{2},\quad i\in\mathbb{N}.
\]
It follows that $x_{i+1}-x_{i}=i$, and so 
\[
\left\Vert \delta_{x_{i}}\right\Vert _{\mathscr{H}}^{2}=\frac{x_{i+1}-x_{i}}{\left(x_{i}-x_{i-1}\right)\left(x_{i+1}-x_{i}\right)}=\frac{2i-1}{\left(i-1\right)i}\xrightarrow[i\rightarrow\infty]{}0.
\]
\textbf{Conclusion.} $\left\Vert \delta_{x_{i}}\right\Vert _{\mathscr{H}}\xrightarrow[i\rightarrow\infty]{}0$
if the set $V=\left\{ x_{i}\right\} _{i=1}^{\infty}\subset\mathbb{R}_{+}$
is sparse. 
\end{example}
Now, some general facts:
\begin{lem}
Let $k:V\times V\rightarrow\mathbb{C}$ be p.d., and let $\mathscr{H}$
be the corresponding RKHS. If $x_{1}\in V$, and if $\delta_{x_{1}}$
has a representation as follows:
\begin{equation}
\delta_{x_{1}}=\sum_{y\in V}\zeta^{\left(x_{1}\right)}\left(y\right)k_{y}\;,\label{eq:pr1}
\end{equation}
then
\begin{equation}
\left\Vert \delta_{x_{1}}\right\Vert _{\mathscr{H}}^{2}=\zeta^{\left(x_{1}\right)}\left(x_{1}\right).\label{eq:pr2}
\end{equation}
\end{lem}
\begin{proof}
Substitute both sides of (\ref{eq:pr1}) into $\left\langle \delta_{x_{1}},\cdot\right\rangle _{\mathscr{H}}$
where $\left\langle \cdot,\cdot\right\rangle _{\mathscr{H}}$ denotes
the inner product in $\mathscr{H}$. 
\end{proof}
\textbf{Application.} Suppose $V=\cup_{n}F_{n}$, $F_{n}\subset F_{n+1}$,
where each $F_{n}\in\mathscr{F}\left(V\right)$, then if $x_{1}\in F_{n}$,
we have 
\begin{equation}
P_{F_{n}}\left(\delta_{x_{1}}\right)=\sum_{y\in F_{n}}\left\langle x_{1},K_{F_{n}}^{-1}y\right\rangle _{l^{2}}k_{y}\label{eq:pr3}
\end{equation}
and 
\begin{equation}
\left\Vert P_{F_{n}}\left(\delta_{x_{1}}\right)\right\Vert _{\mathscr{H}}^{2}=\left\langle x_{1},K_{F_{n}}^{-1}x_{1}\right\rangle _{l^{2}}=\left(K_{F_{n}}^{-1}\delta_{x_{1}}\right)\left(x_{1}\right)\label{eq:pr4}
\end{equation}
and the expression $\left\Vert P_{F_{n}}\left(\delta_{x_{1}}\right)\right\Vert _{\mathscr{H}}^{2}$
is monotone in $n$, i.e., 
\[
\left\Vert P_{F_{n}}\left(\delta_{x_{1}}\right)\right\Vert _{\mathscr{H}}^{2}\leq\left\Vert P_{F_{n+1}}\left(\delta_{x_{1}}\right)\right\Vert _{\mathscr{H}}^{2}\leq\cdots\leq\left\Vert \delta_{x_{1}}\right\Vert _{\mathscr{H}}
\]
with 
\[
\sup_{n\in\mathbb{N}}\left\Vert P_{F_{n}}\left(\delta_{x_{1}}\right)\right\Vert _{\mathscr{H}}^{2}=\lim_{n\rightarrow\infty}\left\Vert P_{F_{n}}\left(\delta_{x_{1}}\right)\right\Vert _{\mathscr{H}}^{2}=\left\Vert \delta_{x_{1}}\right\Vert _{\mathscr{H}}.
\]

\begin{question}
Let $k:\mathbb{R}^{d}\times\mathbb{R}^{d}\rightarrow\mathbb{R}$ be
positive definite, and let $V\subset\mathbb{R}^{d}$ be a countable
discrete subset, e.g., $V=\mathbb{Z}^{d}$. When does $k\big|_{V\times V}$
have the \uline{discrete mass} property? 
\end{question}
Examples of the affirmative, or not, will be discussed below.

\subsection{Discrete RKHS from restrictions}

Let $D:=[0,\infty)$, and $k:D\times D\rightarrow\mathbb{R}$, with
\[
k\left(x,y\right)=x\wedge y=\min\left(x,y\right).
\]
Restrict to $V:=\left\{ 0\right\} \cup\mathbb{Z}_{+}\subset D$, i.e.,
consider 
\[
k^{\left(V\right)}=k\big|_{V\times V}.
\]
$\mathscr{H}\left(k\right)$: Cameron-Martin Hilbert space, consisting
of functions $f\in L^{2}\left(\mathbb{R}\right)$ s.t. 
\[
\int_{0}^{\infty}\left|f'\left(x\right)\right|^{2}dx<\infty,\quad f\left(0\right)=0.
\]
$\mathscr{H}_{V}:=\mathscr{H}\left(k_{V}\right)$. Note that 
\[
f\in\mathscr{H}\left(k_{V}\right)\Longleftrightarrow\sum_{n}\left|f\left(n\right)-f\left(n+1\right)\right|^{2}<\infty.
\]

\begin{lem}
We have $\delta_{n}=2k_{n}-k_{n+1}-k_{n-1}$. \end{lem}
\begin{proof}
Introduce the discrete Laplacian $\Delta$, where 
\[
\left(\Delta f\right)\left(n\right)=2f\left(n\right)-f\left(n-1\right)-f\left(n+1\right),
\]
then $\Delta k_{n}=\delta_{n}$, and 
\[
\left\langle 2k_{n}-k_{n+1}-k_{n-1},k_{m}\right\rangle _{\mathscr{H}_{V}}=\left\langle \delta_{n},k_{m}\right\rangle _{\mathscr{H}_{V}}=\delta_{n,m}.
\]
\end{proof}
\begin{rem}
The same argument as in the proof of the lemma shows (\emph{mutatis
mutandis}) that any ordered discrete countable infinite subset $V\subset[0,\infty)$
yields 
\[
\mathscr{H}_{V}:=\mathscr{H}\left(k\big|_{V\times V}\right)
\]
as a RKHS which is discrete in that (Def. \ref{def:dmp}) if $V=\left\{ x_{i}\right\} _{i=1}^{\infty}$,
$x_{i}\in\mathbb{R}_{+}$, then $\delta_{x_{i}}\in\mathscr{H}_{V}$,
$\forall i\in\mathbb{N}$. \end{rem}
\begin{proof}
Fix vertices $V=\left\{ x_{i}\right\} _{i=1}^{\infty}$, 
\begin{equation}
0<x_{1}<x_{2}<\cdots<x_{i}<x_{i+1}<\infty,\quad x_{i}\rightarrow\infty.
\end{equation}
Assign conductance 
\begin{equation}
c_{i,i+1}=c_{i+1,i}=\frac{1}{x_{i+1}-x_{i}}\left(=\frac{1}{\text{dist}}\right)
\end{equation}
Let 
\begin{eqnarray}
\left(\Delta f\right)\left(x_{i}\right) & = & \left(\frac{1}{x_{i+1}-x_{i}}+\frac{1}{x_{i}-x_{i-1}}\right)f\left(x_{i}\right)\nonumber \\
 &  & -\frac{1}{x_{i}-x_{i-1}}f\left(x_{i-1}\right)-\frac{1}{x_{i+1}-x_{i}}f\left(x_{i+1}\right)
\end{eqnarray}
Equivalently, 
\begin{equation}
\left(\Delta f\right)\left(x_{i}\right)=\left(c_{i,i+1}+c_{i,i-1}\right)f\left(x_{i}\right)-c_{i,i-1}f\left(x_{i-1}\right)-c_{i,i+1}f\left(x_{i+1}\right).\label{eq:glap}
\end{equation}

\begin{rem}
The most general graph-Laplacians will be discussed in detail in sect.
\ref{sec:net} below.
\end{rem}
Then, with (\ref{eq:glap}) we have: 
\[
\Delta k_{x_{i}}=\delta_{x_{i}}
\]
where $k\left(\cdot,\cdot\right)=$ restriction of $s\wedge t$ from
$[0,\infty)\times[0,\infty)$ to $V\times V$; and therefore 
\begin{equation}
\delta_{x_{i}}=\left(c_{i,i+1}+c_{i,i-1}\right)k_{x_{i}}-c_{i,i+1}k_{x_{i+1}}-c_{i,i-1}k_{x_{i-1}}\in\mathscr{H}_{V}
\end{equation}
as the RHS in the last equation is a finite sum. Note that now the
RKHS is 
\[
\mathscr{H}_{V}=\left\{ f:V\rightarrow\mathbb{C}\:\big|\:\sum_{i=1}^{\infty}c_{i,i+1}\left|f\left(x_{i+1}\right)-f\left(x_{i}\right)\right|^{2}<\infty\right\} .
\]

\end{proof}

\subsection{The Brownian bridge}

Let $D:=\left(0,1\right)=$ the open interval $0<t<1$, and set 
\begin{equation}
k_{bridge}\left(s,t\right):=s\wedge t-st;\label{eq:bb1}
\end{equation}
then (\ref{eq:bb1}) is the covariance function for the Brownian bridge
$B_{bri}\left(t\right)$, i.e., 
\begin{equation}
B_{bri}\left(0\right)=B_{bri}\left(1\right)=0\label{eq:bb2}
\end{equation}

\begin{figure}[H]
\includegraphics[width=0.5\columnwidth]{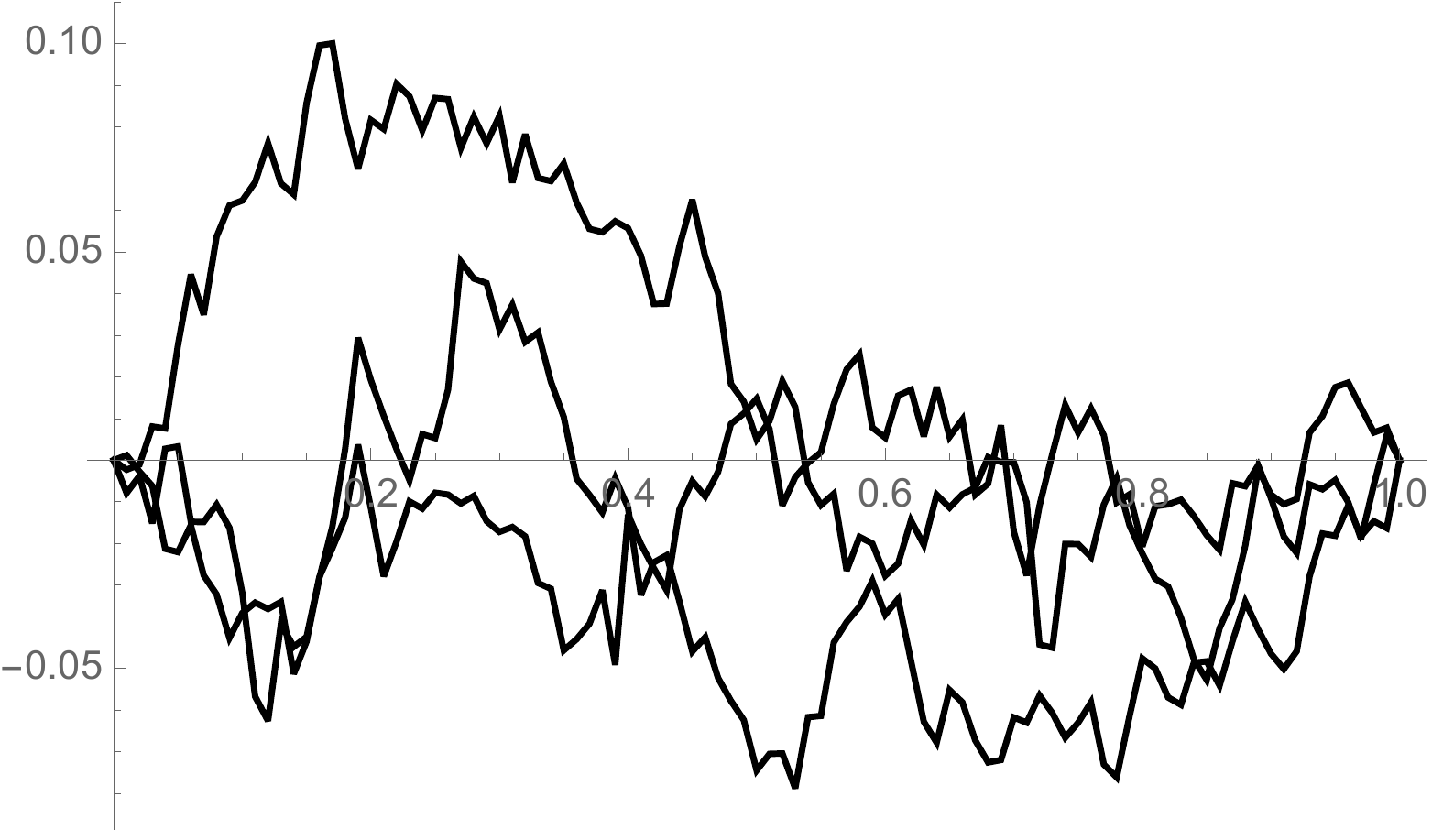}

\protect\caption{\label{fig:bb}Brownian bridge $B_{bri}\left(t\right)$, a simulation
of three sample paths of the Brownian bridge.}

\end{figure}

\begin{equation}
B_{bri}\left(t\right)=\left(1-t\right)B\left(\frac{t}{1-t}\right),\quad0<t<1;\label{eq:bb3}
\end{equation}
where $B\left(t\right)$ is Brownian motion; see Lemma \ref{lem:mc1}.

The corresponding Cameron-Martin space is now 
\begin{equation}
\mathscr{H}_{bri}=\left\{ f\;\mbox{on}\:\left[0,1\right];f'\in L^{2}\left(0,1\right),f\left(0\right)=f\left(1\right)=0\right\} \label{eq:bb4}
\end{equation}
with 
\begin{equation}
\left\Vert f\right\Vert _{\mathscr{H}_{bri}}^{2}:=\int_{0}^{1}\left|f'\left(s\right)\right|^{2}ds<\infty.\label{eq:bb5}
\end{equation}

If $V=\left\{ x_{i}\right\} _{i=1}^{\infty}$, $x_{1}<x_{2}<\cdots<1$,
is the discrete subset of $D$, then we have for $F_{n}\in\mathscr{F}\left(V\right)$,
$F_{n}=\left\{ x_{1},x_{2},\cdots,x_{n}\right\} $, 
\begin{equation}
K_{F_{n}}=\left(k_{bridge}\left(x_{i},x_{j}\right)\right)_{i,j=1}^{n},\label{eq:bb6}
\end{equation}
see (\ref{eq:bb1}), and 
\begin{equation}
\det K_{F_{n}}=x_{1}\left(x_{2}-x_{1}\right)\cdots\left(x_{n}-x_{n-1}\right)\left(1-x_{n}\right).\label{eq:bb7}
\end{equation}

As a result, we get $\delta_{x_{i}}\in\mathscr{H}_{V}^{\left(bri\right)}$
for all $i$, and 
\[
\left\Vert \delta_{x_{i}}\right\Vert _{\mathscr{H}_{V}^{\left(bri\right)}}^{2}=\frac{x_{i+1}-x_{i-1}}{\left(x_{i+1}-x_{i}\right)\left(x_{i}-x_{i-1}\right)}.
\]
Note $\lim_{x_{i}\rightarrow1}\left\Vert \delta_{x_{i}}\right\Vert _{\mathscr{H}_{V}^{\left(bri\right)}}^{2}=\infty$.

\subsection{Binomial RKHS}
\begin{defn}
Let $V=\mathbb{Z}_{+}\cup\left\{ 0\right\} $; and 
\[
k_{b}\left(x,y\right):=\sum_{n=0}^{x\wedge y}\binom{x}{n}\binom{y}{n},\quad\left(x,y\right)\in V\times V.
\]
where $\binom{x}{n}=\frac{x\left(x-1\right)\cdots\left(x-n+1\right)}{n!}$
denotes the standard binomial coefficient from the binomial expansion.

Let $\mathscr{H}=\mathscr{H}\left(k_{b}\right)$ be the corresponding
RKHS. Set 
\begin{equation}
e_{n}\left(x\right)=\begin{cases}
\binom{x}{n} & \text{if \ensuremath{n\leq x}}\\
0 & \text{if \ensuremath{n>x}}.
\end{cases}\label{eq:b1}
\end{equation}
\end{defn}
\begin{lem}[see \cite{AJ15}]
\label{lem:b1}~
\begin{enumerate}[label=(\roman{enumi})]
\item $e_{n}\left(\cdot\right)\in\mathscr{H}$, $n\in V$; 
\item $\left\{ e_{n}\right\} _{n\in V}$ is an orthonormal basis (ONB) in
the Hilbert space $\mathscr{H}$. 
\item Set $F_{n}=\left\{ 0,1,2,\ldots,n\right\} $, and 
\begin{equation}
P_{F_{n}}=\sum_{k=0}^{n}\left|e_{k}\left\rangle \right\langle e_{k}\right|\label{eq:b2}
\end{equation}
or equivalently 
\begin{equation}
P_{F_{n}}f=\sum_{k=0}^{n}\left\langle e_{k},f\right\rangle _{\mathscr{H}}e_{k}\,.\label{eq:b3}
\end{equation}

\end{enumerate}

then, 
\begin{enumerate}[resume]
\item[(iv)] Formula (\ref{eq:b3}) is well defined for all functions $f:V\rightarrow\mathbb{C}$,
$f\in\mathscr{F}unc\left(V\right)$. \end{enumerate}
\begin{enumerate}
\item[(v)] Given $f\in\mathscr{F}unc\left(V\right)$; then 
\begin{equation}
f\in\mathscr{H}\Longleftrightarrow\sum_{k=0}^{\infty}\left|\left\langle e_{k},f\right\rangle _{\mathscr{H}}\right|^{2}<\infty;\label{eq:b4}
\end{equation}
and, in this case, 
\[
\left\Vert f\right\Vert _{\mathscr{H}}^{2}=\sum_{k=0}^{\infty}\left|\left\langle e_{k},f\right\rangle _{\mathscr{H}}\right|^{2}.
\]

\end{enumerate}
\end{lem}

Fix $x_{1}\in V$, then we shall apply Lemma \ref{lem:b1} to the
function $f_{1}=\delta_{x_{1}}$ (in $\mathscr{F}unc\left(V\right)$),
$f_{1}\left(y\right)=\begin{cases}
1 & \text{if \ensuremath{y=x_{1}}}\\
0 & \text{if \ensuremath{y\neq x_{1}}.}
\end{cases}$ 
\begin{thm}
\label{thm:bino}We have 
\[
\left\Vert P_{F_{n}}\left(\delta_{x_{1}}\right)\right\Vert _{\mathscr{H}}^{2}=\sum_{k=x_{1}}^{n}\binom{k}{x_{1}}^{2}.
\]

\end{thm}
The proof of the theorem will be subdivided in steps; see below. 
\begin{lem}[see \cite{AJ15}]
~
\begin{enumerate}[label=(\roman{enumi})]
\item \label{enu:b1}For $\forall m,n\in V$, such that $m\leq n$, we
have 
\begin{equation}
\delta_{m,n}=\sum_{j=m}^{n}\left(-1\right)^{m+j}\binom{n}{j}\binom{j}{m}.\label{eq:b5}
\end{equation}

\item \label{enu:b2}For all $n\in\mathbb{Z}_{+}$, the inverse of the following
lower triangle matrix is this: With (see Fig \ref{fig:L}) 
\begin{equation}
L_{xy}^{\left(n\right)}=\begin{cases}
\binom{x}{y} & \text{if \ensuremath{y\leq x\leq n}}\\
0 & \text{if \ensuremath{x<y}}
\end{cases}\label{eq:b6}
\end{equation}
 we have:
\begin{equation}
\left(L^{\left(n\right)}\right)_{xy}^{-1}=\begin{cases}
\left(-1\right)^{x-y}\binom{x}{y} & \text{if \ensuremath{y\leq x\leq n}}\\
0 & \text{if \ensuremath{x<y}}.
\end{cases}\label{eq:b7}
\end{equation}
 
\end{enumerate}

Notation: The numbers in (\ref{eq:b7}) are the entries of the matrix
$\left(L^{\left(n\right)}\right)^{-1}$. 

\end{lem}
\begin{proof}
We refer to \cite{AJ15}. In rough outline, \ref{enu:b2} follows
from \ref{enu:b1}.
\end{proof}
\begin{figure}[H]
\[
L^{\left(n\right)}=\begin{bmatrix}1 & 0 & 0 & 0 & \cdots & \cdots & 0 & \cdots & 0 & 0\\
1 & 1 & 0 & 0 & \cdots & \cdots & 0 & \cdots & 0 & 0\\
1 & 2 & 1 & 0 &  &  & \vdots &  & \vdots & \vdots\\
1 & 3 & 3 & 1 & \ddots &  & \vdots &  & \vdots & \vdots\\
\vdots & \vdots & \vdots & \vdots & \ddots &  & \vdots &  & \vdots & \vdots\\
\vdots & \vdots & \vdots & \vdots &  & 1 & 0 &  & \vdots & \vdots\\
1 & \cdots & \binom{x}{y} & \binom{x}{y+1} & \cdots & * & 1 & \ddots & \vdots & \vdots\\
\vdots & \vdots & \vdots & \vdots &  &  &  & \ddots & 0 & \vdots\\
\vdots & \vdots & \vdots & \vdots &  &  &  &  & 1 & 0\\
1 & \cdots & \binom{n}{y} & \binom{n}{y+1} & \cdots & \cdots & \cdots & \cdots & n & 1
\end{bmatrix}
\]

\protect\caption{\label{fig:L}The matrix $L_{n}$ is simply a truncated Pascal triangle,
arranged to fit into a lower triangular matrix.}
\end{figure}

\begin{cor}
\label{cor:bino}Let $k_{b}$, $\mathscr{H}$, and $n\in\mathbb{Z}_{+}$
be as above with the lower triangle matrix $L_{n}$. Set 
\begin{equation}
K_{n}\left(x,y\right)=k_{b}\left(x,y\right),\quad\left(x,y\right)\in F_{n}\times F_{n},\label{eq:b8}
\end{equation}
i.e., an $\left(n+1\right)\times\left(n+1\right)$ matrix. 
\begin{enumerate}[label=(\roman{enumi})]
\item Then $K_{n}$ is invertible with 
\begin{equation}
K_{n}^{-1}=\left(L_{n}^{tr}\right)^{-1}\left(L_{n}\right)^{-1};\label{eq:b9}
\end{equation}
an $(\text{upper triangle})\times(\text{lower triangle})$ factorization. 
\item For the diagonal entries in the $\left(n+1\right)\times\left(n+1\right)$
matrix $K_{n}^{-1}$, we have:
\[
\left\langle x,K_{n}^{-1}x\right\rangle _{l^{2}}=\sum_{k=x}^{n}\binom{k}{x}^{2}
\]

\end{enumerate}

\textbf{Conclusion.} Since 
\begin{equation}
\left\Vert P_{F_{n}}\left(\delta_{x_{1}}\right)\right\Vert _{\mathscr{H}}^{2}=\left\langle x_{1},K_{n}^{-1}x_{1}\right\rangle _{\mathscr{H}}\label{eq:b11}
\end{equation}
for all $x_{1}\in F_{n}$, we get 
\begin{eqnarray}
\left\Vert P_{F_{n}}\left(\delta_{x_{1}}\right)\right\Vert _{\mathscr{H}}^{2} & = & \sum_{k=x_{1}}^{n}\binom{k}{x_{1}}^{2}\nonumber \\
 & = & 1+\binom{x_{1}+1}{x_{1}}^{2}+\binom{x_{1}+2}{x_{1}}^{2}+\cdots+\binom{n}{x_{1}}^{2};\label{eq:b12}
\end{eqnarray}
and therefore, 
\[
\left\Vert \delta_{x_{1}}\right\Vert _{\mathscr{H}}^{2}=\sum_{k=x_{1}}^{\infty}\binom{k}{x_{1}}^{2}=\infty.
\]
In other words, no $\delta_{x}$ is in $\mathscr{H}$.

\end{cor}

\section{\label{sec:net}Infinite network of resistors}

Here we introduce a family of positive definite kernels $k:V\times V\rightarrow\mathbb{R}$,
defined on infinite sets $V$ of vertices for a given graph $G=\left(V,E\right)$
with edges $E\subset V\times V\backslash(\text{diagonal})$.

There is a large literature dealing with analysis on infinite graphs;
see e.g., \cite{JP10,JP11,JP13}; see also \cite{OS05,BCF07,CJ11}.

Our main purpose here is to point out that every assignment of resistors
on the edges $E$ in $G$ yields a p.d. kernel $k$, and an associated
RKHS $\mathscr{H}=\mathscr{H}\left(k\right)$ such that
\begin{equation}
\delta_{x}\in\mathscr{H},\quad\text{for all \ensuremath{x\in V}.}\label{eq:g1}
\end{equation}

\begin{defn}
\label{def:g}Let $G=\left(V,E\right)$ be as above. Assume 
\begin{enumerate}[label=\arabic{enumi}.]
\item $\left(x,y\right)\in E\Longleftrightarrow\left(y,x\right)\in E$;
\item $\exists c:E\rightarrow\mathbb{R}_{+}$ (a conductance function =
1 / resistance) such that

\begin{enumerate}[label=(\roman{enumii})]
\item \label{enu:g1} $c_{\left(xy\right)}=c_{\left(yx\right)}$, $\forall\left(xy\right)\in E$; 
\item \label{enu:g2}for all $x\in V$, $\#\left\{ y\in V\:|\: c_{\left(xy\right)}>0\right\} <\infty$;
and
\item $\exists o\in V$ s.t. for $\forall x\in V\backslash\left\{ o\right\} $,
$\exists$ edges $\left(x_{i},x_{i+1}\right)_{0}^{n-1}\in E$ s.t.
$x_{o}=0$, and $x_{n}=x$; called connectedness. 
\end{enumerate}
\end{enumerate}
\end{defn}
Given $G=\left(V,E\right)$, and a fixed conductance function $c:E\rightarrow\mathbb{R}_{+}$
as specified above, we now define a corresponding Laplace operator
$\Delta=\Delta^{\left(c\right)}$ acting on functions on $V$, i.e.,
on $\mathscr{F}unc\left(V\right)$ by 
\begin{equation}
\left(\Delta f\right)\left(x\right)=\sum_{y\sim x}c_{xy}\left(f\left(x\right)-f\left(y\right)\right).\label{eq:g2}
\end{equation}

Let $\mathscr{H}$ be the Hilbert space defined as follows: A function
$f$ on $V$ is in $\mathscr{H}$ iff $f\left(o\right)=0$, and 
\begin{equation}
\left\Vert f\right\Vert _{\mathscr{H}}^{2}:=\frac{1}{2}\underset{\underset{\subset V\times V}{\left(x,y\right)\in E}}{\sum\sum}c_{xy}\left|f\left(x\right)-f\left(y\right)\right|^{2}<\infty.\label{eq:g3}
\end{equation}

\begin{lem}[\cite{JP10}]
\label{lem:lap2}For all $x\in V\backslash\left\{ o\right\} $, $\exists v_{x}\in\mathscr{H}$
s.t. 
\begin{equation}
f\left(x\right)-f\left(o\right)=\left\langle v_{x},f\right\rangle _{\mathscr{H}},\quad\forall f\in\mathscr{H}\label{eq:g4}
\end{equation}
where 
\begin{equation}
\left\langle h,f\right\rangle _{\mathscr{H}}=\frac{1}{2}\underset{\left(x,y\right)\in E}{\sum\sum}c_{xy}\left(\overline{h\left(x\right)}-\overline{h\left(y\right)}\right)\left(f\left(x\right)-f\left(y\right)\right),\quad\forall h,f\in\mathscr{H}.\label{eq:g41}
\end{equation}
(The system $\left\{ v_{x}\right\} $ is called a system of \uline{dipoles}.
)\end{lem}
\begin{proof}
Let $x\in V\backslash\left\{ o\right\} $, and use (\ref{eq:g2})
together with the Schwarz-inequality to show that 
\[
\left|f\left(x\right)-f\left(o\right)\right|^{2}\leq\sum_{i}\frac{1}{c_{x_{i}x_{i+1}}}\sum_{i}c_{x_{i}x_{i+1}}\left|f\left(x_{i}\right)-f\left(x_{i+1}\right)\right|^{2}.
\]
An application of Riesz' lemma then yields the desired conclusion. 

Note that $v_{x}=v_{x}^{\left(c\right)}$ depends on the choice of
base point $o\in V$, and on conductance function $c$; see \ref{enu:g1}-\ref{enu:g2}
and (\ref{eq:g3}).
\end{proof}
Now set 
\begin{equation}
k^{\left(c\right)}\left(x,y\right)=\left\langle v_{x},v_{y}\right\rangle _{\mathscr{H}},\quad\forall\left(xy\right)\in\left(V\backslash\left\{ o\right\} \right)\times\left(V\backslash\left\{ o\right\} \right).\label{eq:g5}
\end{equation}
It follows from a theorem that $k^{\left(c\right)}$ is a Green's
function for the Laplacian $\Delta^{\left(c\right)}$ in the sense
that 
\begin{equation}
\Delta^{\left(c\right)}k^{\left(c\right)}\left(x,\cdot\right)=\delta_{x}\label{eq:g6}
\end{equation}
where the dot in (\ref{eq:g6}) is the dummy-variable in the action.
(Note that the solution to (\ref{eq:g6}) is not unique.)
\begin{lem}[\cite{JP11}]
Let $G=\left(V,E\right)$, and conductance function $c:E\rightarrow\mathbb{R}_{+}$
be a s specified above; then $k^{\left(c\right)}$ in (\ref{eq:g5})
is positive definite, and the corresponding RKHS $\mathscr{H}\left(k^{\left(c\right)}\right)$
is the Hilbert space introduced in (\ref{eq:g3}) and (\ref{eq:g41}),
called the energy-Hilbert space. \end{lem}
\begin{proof}
See \cite{JP10,JP11,JP13}.\end{proof}
\begin{prop}
Let $x\in V\backslash\left\{ o\right\} $, and let $c:E\rightarrow\mathbb{R}_{+}$
be specified as above. Let $\mathscr{H}=\mathscr{H}\left(k^{c}\right)$
be the corresponding RKHS. Then $\delta_{x}\in\mathscr{H}$, and 
\begin{equation}
\left\Vert \delta_{x}\right\Vert _{\mathscr{H}}^{2}=\sum_{y\sim x}c_{\left(xy\right)}=:c\left(x\right).\label{eq:g7}
\end{equation}
\end{prop}
\begin{proof}
We study the finite matrices, defined for $\forall F\in\mathscr{F}\left(V\right)$,
by 
\begin{equation}
K_{F}\left(x,y\right)=k^{c}\left(x,y\right),\quad\left(x,y\right)\in F\times F.\label{eq:g8}
\end{equation}
Fix $x\in V\backslash\left\{ o\right\} $, and pick $F\in\mathscr{F}\left(V\right)$
such that 
\begin{equation}
\left\{ x\right\} \cup\left\{ y\in V\:|\: y\sim x\right\} \subset F,\label{eq:g9}
\end{equation}
see Fig \ref{fig:nb}; an interior point:

\begin{figure}[H]
\includegraphics[width=0.6\textwidth]{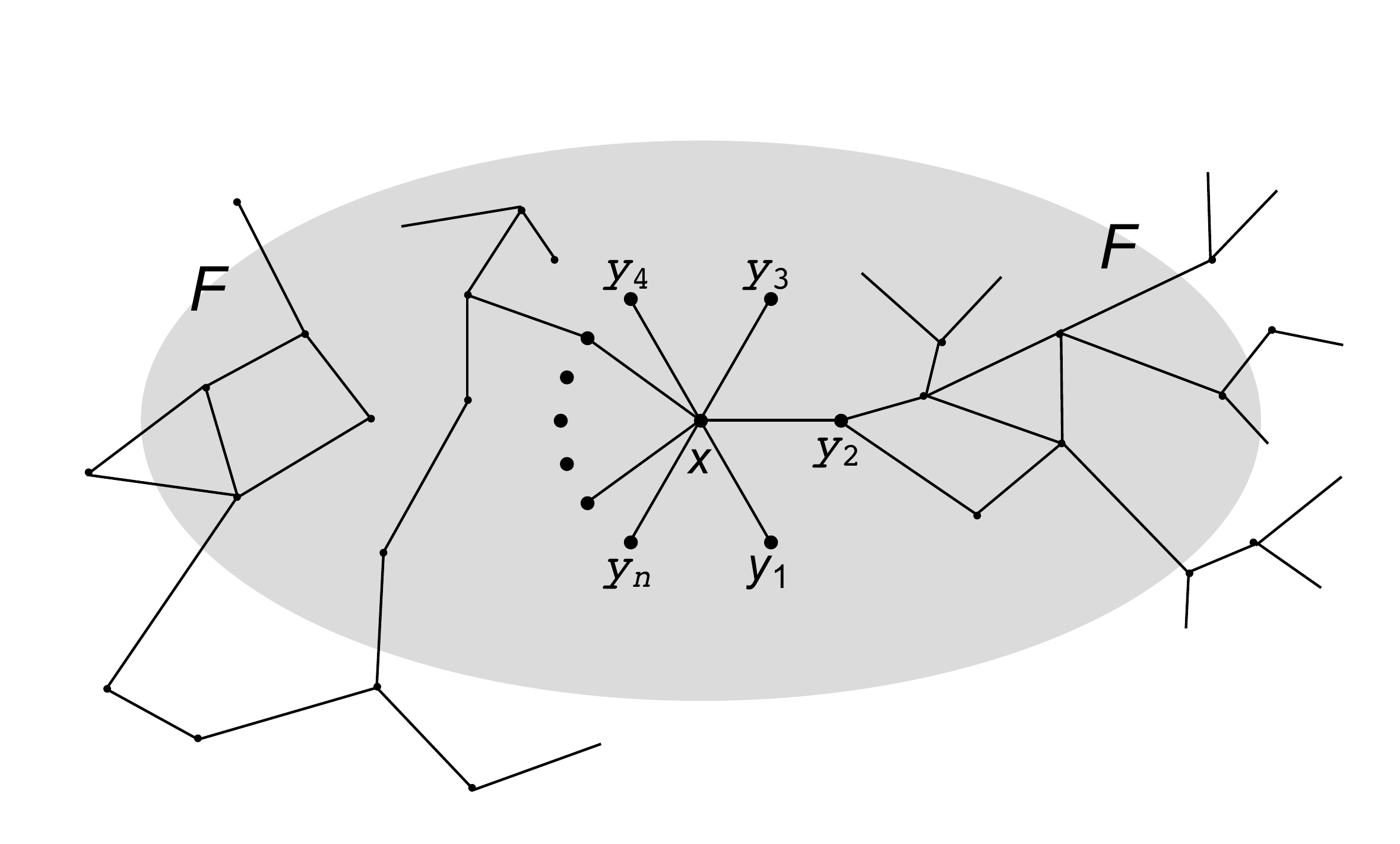}

\protect\caption{\label{fig:nb}Neighborhood of $x$, see Def. \ref{def:g} \ref{enu:g2}.
An interior point $x$. }
\end{figure}

Let $F\in\mathscr{F}\left(V\right)$ be as in (\ref{eq:g8}) and in
Fig \ref{fig:nb}, and let $\Delta=\Delta^{\left(c\right)}$ be the
Laplace operator (\ref{eq:g2}), then for all $\left(x,y\right)\in F\times F$,
we have:
\begin{eqnarray}
\left\langle x,K_{F}^{-1}y\right\rangle _{l^{2}} & = & \left\langle \delta_{x},\Delta\delta_{y}\right\rangle _{l^{2}}\nonumber \\
 & = & \left(\Delta\delta_{y}\right)\left(x\right)\nonumber \\
 & = & \begin{cases}
c\left(x\right) & \text{if \ensuremath{y=x}; see }\left(\ref{eq:g7}\right)\\
-c_{\left(xy\right)} & \text{if \ensuremath{y\sim x}}\\
0 & \text{for all other values of \ensuremath{y}}
\end{cases}\label{eq:g10}
\end{eqnarray}
In particular, 
\[
\sup_{F\in\mathscr{F}\left(V\right)}\left(K_{F}\delta_{x}\right)\left(x\right)<\infty;
\]
and in fact, 
\[
\left\Vert \delta_{x}\right\Vert _{\mathscr{H}}^{2}=c\left(x\right),\;\text{for all \ensuremath{x\in V\backslash\left\{ o\right\} },}
\]
as claimed in the Proposition. 

The last step in the present proof uses the equivalence \ref{enu:d1}$\Leftrightarrow$\ref{enu:d2}$\Leftrightarrow$\ref{enu:d3}
from Theorem \ref{thm:del} above. 

Finally, we note that the assertion in (\ref{eq:g10}) follows from
\begin{equation}
\Delta v_{x}=\delta_{x}-\delta_{o},\quad\forall x\in V\backslash\left\{ o\right\} .\label{eq:g11}
\end{equation}
And (\ref{eq:g11}) in turn follows from (\ref{eq:g4}), (\ref{eq:g2})
and a straightforward computation.\end{proof}
\begin{cor}
\label{cor:lap1}Let $G=\left(V,E\right)$ and conductance $c:E\rightarrow\mathbb{R}_{+}$
be as specified above. Let $\Delta=\Delta^{\left(c\right)}$ be the
corresponding Laplace operator. Let $\mathscr{H}=\mathscr{H}\left(k^{c}\right)$
be the RKHS. Then 
\begin{equation}
\left\langle \delta_{x},f\right\rangle _{\mathscr{H}}=\left(\Delta f\right)\left(x\right)\label{eq:g12}
\end{equation}
and 
\begin{equation}
\delta_{x}=c\left(x\right)v_{x}-\sum_{y\sim x}c_{xy}v_{y}\label{eq:g121}
\end{equation}
holds for all $x\in V$. \end{cor}
\begin{proof}
Since the system $\left\{ v_{x}\right\} $ of dipoles (see (\ref{eq:g4}))
span a dense subspace in $\mathscr{H}$, it is enough to verify (\ref{eq:g12})
when $f=v_{y}$ for $y\in V\backslash\left\{ o\right\} $. But in
this case, (\ref{eq:g12}) follows from (\ref{eq:g6}) and (\ref{eq:g10}).\end{proof}
\begin{cor}
\label{cor:lap4}Let $G=\left(V,E\right)$, and conductance $c:E\rightarrow\mathbb{R}_{+}$
be as before; let $\Delta^{\left(c\right)}$ be the Laplace operator,
and $\mathscr{H}_{E}^{\left(c\right)}$ the energy-Hilbert space in
Definition \ref{def:g} (see (\ref{eq:g3})). Let $k^{\left(c\right)}\left(x,y\right)=\left\langle v_{x},v_{y}\right\rangle _{\mathscr{H}_{E}}$
be the kernel from (\ref{eq:g5}), i.e., the Green's function of $\Delta^{\left(c\right)}$.
Then the two Hilbert spaces $\mathscr{H}_{E}$, and $\mathscr{H}\left(k^{\left(c\right)}\right)=RKHS\left(k^{\left(c\right)}\right)$,
are naturally isometrically isomorphic via $v_{x}\longmapsto k_{x}^{\left(c\right)}$
where $k_{x}^{\left(c\right)}=k^{\left(c\right)}\left(x,\cdot\right)$
for all $x\in V$. \end{cor}
\begin{proof}
Let $F\in\mathscr{F}\left(V\right)$, and let $\xi$ be a function
on $F$; then 
\begin{eqnarray*}
\left\Vert \sum\nolimits _{x\in F}\xi\left(x\right)k_{x}^{\left(c\right)}\right\Vert _{\mathscr{H}\left(k^{\left(c\right)}\right)}^{2} & = & \underset{F\times F}{\sum\sum}\overline{\xi\left(x\right)}\xi\left(y\right)k^{\left(c\right)}\left(x,y\right)\\
 & \underset{\left(\ref{eq:g5}\right)}{=} & \underset{F\times F}{\sum\sum}\overline{\xi\left(x\right)}\xi\left(y\right)\left\langle v_{x},v_{y}\right\rangle _{\mathscr{H}_{E}}\\
 & = & \left\Vert \sum\nolimits _{x\in F}\xi\left(x\right)v_{x}\right\Vert _{\mathscr{H}_{E}}^{2}.
\end{eqnarray*}

The remaining steps in the proof of the Corollary now follows from
the standard completion from dense subspaces in the respective two
Hilbert spaces $\mathscr{H}_{E}$ and $\mathscr{H}\left(k^{\left(c\right)}\right)$. 
\end{proof}
In the following we show how the kernels $k^{\left(c\right)}:V\times V\rightarrow\mathbb{R}$
from (\ref{eq:g5}) in Lemma \ref{lem:lap2} are related to metrics
on $V$; so called \emph{resistance metrics} (see, e.g., \cite{JP10,AJSV13}.)
\begin{cor}
\label{cor:lap2}Let $G=\left(V,E\right)$, and conductance $c:E\rightarrow\mathbb{R}_{+}$
be as above; and let $k^{\left(c\right)}\left(x,y\right):=\left\langle v_{x},v_{y}\right\rangle _{\mathscr{H}_{E}}$
be the corresponding Green's function for the graph Laplacian $\Delta^{\left(c\right)}$. 

Then there is a \uline{metric} $R\left(=R^{\left(c\right)}=\mbox{the resistance metric}\right)$,
such that 
\begin{equation}
k^{\left(c\right)}\left(x,y\right)=\frac{R^{\left(c\right)}\left(o,x\right)+R^{\left(c\right)}\left(o,y\right)-R^{\left(c\right)}\left(x,y\right)}{2}\label{eq:gm1}
\end{equation}
holds on $V\times V$. Here the base-point $o\in V$ is chosen and
fixed s.t. 
\begin{equation}
\left\langle V_{x},f\right\rangle _{\mathscr{H}_{E}}=f\left(x\right)-f\left(o\right),\quad\forall f\in\mathscr{H}_{E},\;\forall x\in V.\label{eq:gm2}
\end{equation}
\end{cor}
\begin{proof}
See \cite{JP10}. Set 
\begin{equation}
R^{\left(c\right)}\left(x,y\right)=\left\Vert v_{x}-v_{y}\right\Vert _{\mathscr{H}_{E}}^{2}.\label{eq:gm3}
\end{equation}
We proved in \cite{JP10} that $R^{\left(c\right)}\left(x,y\right)$
in (\ref{eq:gm3}) indeed defines a metric on $V$; the so called
\emph{resistance metric}. It represents the voltage-drop from $x$
to $y$ when 1 Amp is fed into $\left(G,c\right)$ at the point $x$,
and then extracted at $y$. 

The verification of (\ref{eq:gm1}) is now an easy computation, as
follows:
\begin{eqnarray*}
 &  & \frac{R^{\left(c\right)}\left(o,x\right)+R^{\left(c\right)}\left(o,y\right)-R^{\left(c\right)}\left(x,y\right)}{2}\\
 & = & \frac{\left\Vert v_{x}\right\Vert _{\mathscr{H}_{E}}^{2}+\left\Vert v_{y}\right\Vert _{\mathscr{H}_{E}}^{2}-\left\Vert v_{x}-v_{y}\right\Vert _{\mathscr{H}_{E}}^{2}}{2}\\
 & = & \left\langle v_{x},v_{y}\right\rangle _{\mathscr{H}_{E}}\\
 & = & k^{\left(c\right)}\left(x,y\right)\quad\mbox{by \ensuremath{\left(\ref{eq:g5}\right)}.}
\end{eqnarray*}
\end{proof}
\begin{prop}
In the two cases: (i) $B\left(t\right)$, Brownian motion on $0<t<\infty$;
and (ii) the Brownian bridge $B_{bri}\left(t\right)$, $0<t<1$, from
sect. \ref{sec:egs}, the corresponding resistance metric $R$ is
as follows:
\begin{enumerate}
\item[(i)]  If $V=\left\{ x_{i}\right\} _{i=1}^{\infty}\subset\left(0,\infty\right)$,
$x_{1}<x_{2}<\cdots$, then 
\begin{equation}
R_{B}^{\left(V\right)}\left(x_{i},x_{j}\right)=\left|x_{i}-x_{j}\right|.
\end{equation}

\item[(ii)]  If $W=\left\{ x_{i}\right\} _{i=1}^{\infty}\subset\left(0,1\right)$,
$0<x_{1}<x_{2}<\cdots<1$, then 
\begin{equation}
R_{bridge}^{\left(W\right)}\left(x_{i},x_{j}\right)=\left|x_{i}-x_{j}\right|\cdot\left(1-\left|x_{i}-x_{j}\right|\right).
\end{equation}
In the completion w.r.t. the resistance metric $R_{bridge}^{\left(W\right)}$,
the two endpoints $x=0$ and $x=1$ are identified; see also Fig \ref{fig:bb}.
\end{enumerate}
\end{prop}

\subsection{Gaussian Processes}
\begin{defn}
A \emph{Gaussian realization} of an infinite graph-network $G=\left(V,E\right)$,
with prescribed conductance function $c:E\rightarrow\mathbb{R}_{+}$,
and \emph{dipoles} $\left(v_{x}^{c}\right)_{x\in V\backslash\left\{ o\right\} }$,
is a Gaussian process $\left(X_{x}\right)_{x\in V}$ on a probability
space $\left(\Omega,\mathscr{F},\mathbb{P}\right)$, where $\Omega$
is a sample space; $\mathscr{F}$ a sigma-algebra of events, and $\mathbb{P}$
a probability measure s.t., for $\forall F\in\mathscr{F}\left(V\right)$,
the random variables $\left(X_{x}\right)_{x\in F}$, are jointly Gaussian
with 
\begin{equation}
\mathbb{E}\left(X_{x}\right)=\int_{\Omega}X_{x}d\mathbb{P}=0
\end{equation}
and covariance 
\begin{equation}
\mathbb{E}\left(X_{x}X_{y}\right)=k^{\left(c\right)}\left(x,y\right)=\left\langle v_{x}^{\left(c\right)},v_{y}^{\left(c\right)}\right\rangle _{\mathscr{H}_{E}};\label{eq:gfc}
\end{equation}
i.e., the covariance matrix $\left(\mathbb{E}\left(X_{x}X_{y}\right)\right)_{\left(x,y\right)\in F\times F}$
is 
\begin{equation}
K_{F}\left(x,y\right):=k^{\left(c\right)}\left(x,y\right)\;\mbox{on \ensuremath{F\times F.}}
\end{equation}
\end{defn}
\begin{lem}[\cite{JP10}]
For all $G=\left(V,E\right)$, and $c:E\rightarrow\mathbb{R}_{+}$,
as specified, Gaussian realizations exist; they are called \uline{Gaussian
free fields}.\end{lem}
\begin{cor}
\label{cor:lap3}Let $G=\left(V,E\right)$, $c:E\rightarrow\mathbb{R}_{+}$
be as above; and let $\left(X_{x}\right)_{x\in V}$ be an associated
Gaussian free field. Then the point Dirac-masses $\left(\delta_{x}\right)_{x\in V}$
have Gaussian realizations
\begin{equation}
\widetilde{\delta_{x}}=c\left(x\right)X_{x}-\sum_{y\sim x}c_{xy}X_{y},\quad\forall x\in V.
\end{equation}

\end{cor}

\begin{cor}
\label{cor:Lap3}Let $G=\left(V,E\right)$, and $c:E\rightarrow\mathbb{R}_{+}$
be as above. Let $\left\{ X_{x}\right\} _{x\in V}$ be the corresponding
Gaussian free field, i.e., with correlation 
\begin{equation}
\mathbb{E}\left(X_{x}X_{y}\right)=k^{\left(c\right)}\left(x,y\right)=\left\langle v_{x}^{\left(c\right)},v_{y}^{\left(c\right)}\right\rangle _{\mathscr{H}_{E}}\label{eq:gf1}
\end{equation}
 where the dipoles $\{v_{x}^{\left(c\right)}\}\subset\mathscr{H}_{E}$
are computed w.r.t. a chosen (and fixed) based-point $o\in V$, i.e.,
\begin{equation}
\left\langle v_{x}^{\left(c\right)},f\right\rangle _{\mathscr{H}_{E}}=f\left(x\right)-f\left(o\right),\quad\forall f\in\mathscr{H}_{E},\; x\in V.\label{eq:gf2}
\end{equation}
Finally, let $R^{\left(c\right)}\left(x,y\right)$ be the corresponding
resistance metric on $V$. Then
\begin{equation}
\mathbb{E}\left(X_{x}X_{z}\right)+\mathbb{E}\left(X_{z}X_{y}\right)\leq\mathbb{E}\left(X_{x}X_{y}\right)+R^{\left(c\right)}\left(o,z\right)\label{eq:gf3}
\end{equation}
holds for all vertices $x,y,z\in V$; see Fig \ref{fig:gf}.\end{cor}
\begin{proof}
Use Corollary \ref{cor:lap2}, and (\ref{eq:gm3}). We have 
\[
\left\Vert v_{x}-v_{y}\right\Vert _{\mathscr{H}}^{2}\leq\left\Vert v_{x}-v_{z}\right\Vert _{\mathscr{H}}^{2}+\left\Vert v_{z}-v_{y}\right\Vert _{\mathscr{H}}^{2},
\]
and (\ref{eq:gf3}) now follows from (\ref{eq:gfc}).
\end{proof}
\begin{figure}[H]
\includegraphics[width=0.5\textwidth]{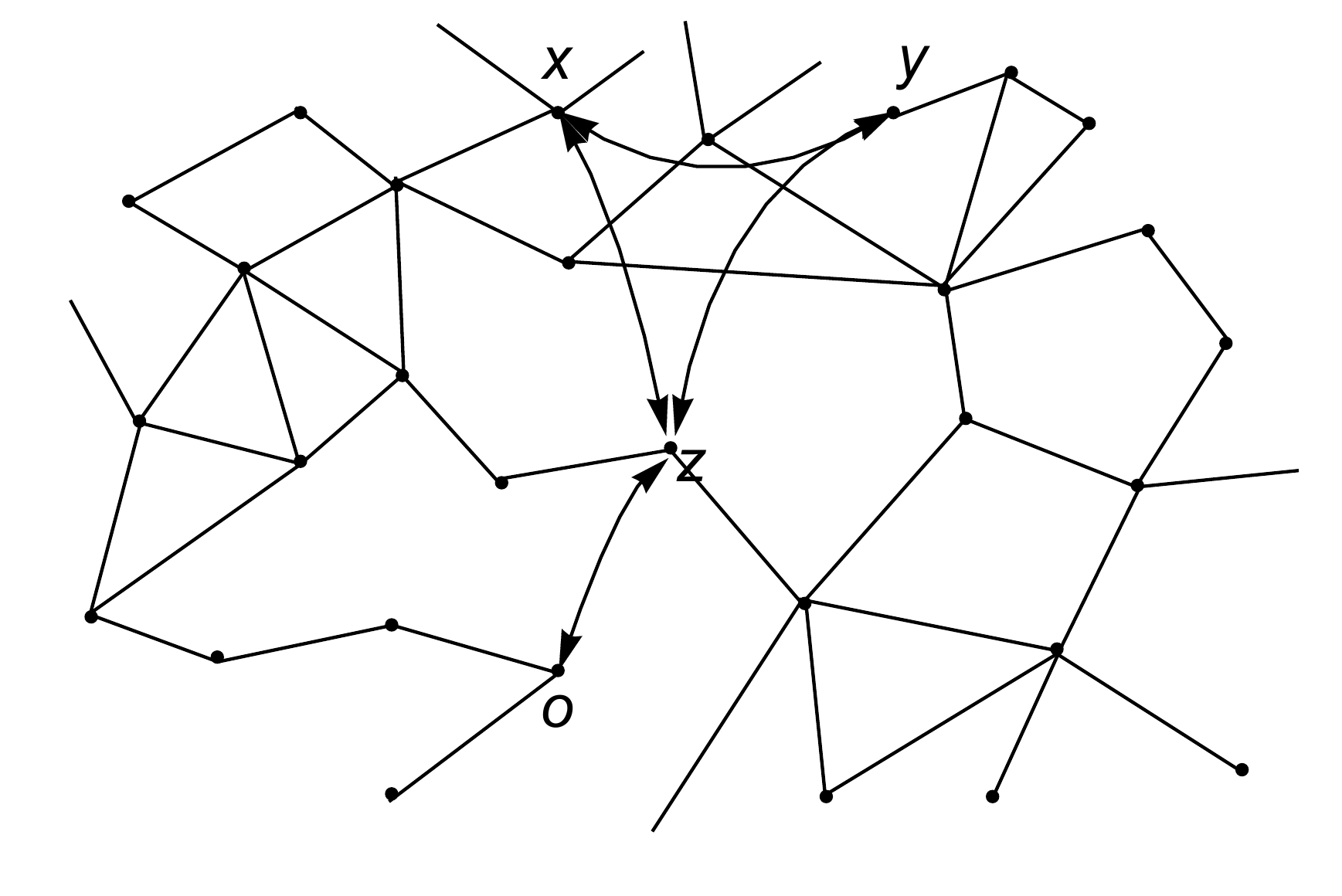}

\protect\caption{\label{fig:gf}Covariance vs resistance distance $R^{\left(c\right)}\left(o,z\right)$
for three vertices $x,y,z\in V$.}
\end{figure}

\subsection{Metric Completion}

The next theorem illustrates a connection between the universal property
of a kernel in a RKHS $\mathscr{H}$, on the one hand, and the distribution
of the Dirac point-masses $\delta_{x}$, on the other. We make \textquotedblleft distribution\textquotedblright{}
precise by the quantity $E\left(x\right):=\Vert\delta_{x}\Vert_{\mathscr{H}}^{2}$,
the energy of the point-mass at the vertex point $x$. We introduce
a metric completion $M$, and the universal property of the RKHS $\mathscr{H}$
asserts that the functions from $\mathscr{H}$ are continuous and
$1/2$-Lipschitz on $M$, and that they approximate every continuous
function on $M$ in the uniform norm. Recall, the vertex set $V$
is equipped with its resistance metric. The universal property here
refers to the corresponding metric completion $M$ of the discrete
vertex set. In the interesting cases (see e.g., Example \ref{exa:btree}),
$M$ is a continuum; -- in the case of the example below, the boundary
of $V$ is a Cantor set. One expects the value of $E\left(x\right)$
to go to infinity as $x$ approaches the boundary $M$, and this is
illustrated in the example; with an explicit formula for $E\left(x\right)$.

Of special interest is the class of networks $(V,E)$ where the resistance
metric $R$ (on the given vertex vertex-set $V$) is bounded; see
\ref{enu:m2} in Theorem \ref{thm:mc} below. This class of networks,
for which the diameter of $V$ measured in the resistance metric $R$
is bounded, includes networks having lots of edges with resistors
occurring in parallel; see e.g., \cite{JP11}.
\begin{thm}
\label{thm:mc}Let $G=\left(V,E\right)$, $c:E\rightarrow\mathbb{R}_{+}$
be as above, and let $R^{\left(c\right)}:V\times V\rightarrow\mathbb{R}_{+}$
be the resistance-metric (see (\ref{eq:gm3})). Let $M$ be the metric
completion of $\left(V,R^{\left(c\right)}\right)$. Then:
\begin{enumerate}[label=(\roman{enumi})]
\item \label{enu:m1} For every $f\in\mathscr{H}$, the function
\begin{equation}
V\ni x\longmapsto f\left(x\right)\in\mathbb{C}\label{eq:m1}
\end{equation}
 extends by closure to a uniformly continuous function $\widetilde{f}:M\mapsto\mathbb{C}$. 
\item \label{enu:m2}If $R^{\left(c\right)}$ is assumed \uline{bounded},
then the RKHS $\mathscr{H}$ is an algebra under point-wise product:
\begin{equation}
\left(f_{1}f_{2}\right)\left(x\right)=f_{1}\left(x\right)f_{2}\left(x\right),\quad f_{i}\in\mathscr{H},\: i=1,2,\: x\in V.\label{eq:m2}
\end{equation}

\item \label{enu:m3}If $M$ is \uline{compact}, then $\{\widetilde{f}\:|\: f\in\mathscr{H}\}$
is dense in $C\left(M\right)$ in the uniform norm.
\end{enumerate}
\end{thm}
\begin{proof}
The assertions in \ref{enu:m1} follow from the following two estimates: 

Let $f\in\mathscr{H}$, then 
\begin{equation}
\left|f\left(x\right)-f\left(y\right)\right|^{2}\leq\left\Vert f\right\Vert _{\mathscr{H}}^{2}R^{\left(c\right)}\left(x,y\right),\quad\forall x,y\in V;\label{eq:m3}
\end{equation}
and 
\begin{equation}
\left|f\left(x\right)\right|\leq\left|f\left(o\right)\right|+R^{\left(c\right)}\left(o,x\right)^{\frac{1}{2}}.\label{eq:m4}
\end{equation}
The estimates in (\ref{eq:m3})-(\ref{eq:m4}), in turn, follow from
Corollaries \ref{cor:lap4} and \ref{cor:lap2}.

To prove \ref{enu:m2}, we compute the energy-norm of the product
$f_{1}\cdot f_{2}$ where $f_{i}\in\mathscr{H}$, $i=1,2$; and we
use Corollary \ref{cor:lap4}:
\begin{eqnarray*}
 &  & \sum_{x}\sum_{y}c_{xy}\left|f_{1}\left(x\right)f_{2}\left(x\right)-f_{1}\left(y\right)f_{2}\left(y\right)\right|^{2}\\
 & = & \sum_{x}\sum_{y}c_{xy}\left|\left(f_{1}\left(x\right)-f_{1}\left(y\right)\right)f_{2}\left(x\right)+f_{1}\left(y\right)\left(f_{2}\left(x\right)-f_{2}\left(y\right)\right)\right|^{2}\\
 & \leq & \sum_{x}\sum_{y}c_{xy}\left(\left|f_{1}\left(x\right)-f_{1}\left(y\right)\right|^{2}+\left|f_{2}\left(x\right)-f_{2}\left(y\right)\right|^{2}\right)\cdot\left(\left|f_{2}\left(x\right)\right|^{2}+\left|f_{1}\left(y\right)\right|^{2}\right)\\
 &  & \left(\text{by Schwarz inside}\right)\\
 & \leq & \left(\left\Vert f_{1}\right\Vert _{\infty}^{2}+\left\Vert f_{2}\right\Vert _{\infty}^{2}\right)\cdot\left(\left\Vert f_{1}\right\Vert _{\mathscr{H}}^{2}+\left\Vert f_{2}\right\Vert _{\mathscr{H}}^{2}\right);
\end{eqnarray*}
and we note that the RHS is finite subject to the assumption in \ref{enu:m2}. 

Proof of \ref{enu:m3}: We are assuming here that $M$ is \emph{compact},
and we shall apply the Stone-Weierstrass theorem to the subalgebra
\begin{equation}
\left\{ \widetilde{f}\:\big|\: f\in\mathscr{H}\right\} \subset C\left(M\right).\label{eq:m5}
\end{equation}
Indeed, the conditions for Stone-Weierstrass are satisfied: The functions
on LHS in (\ref{eq:m5}) form an algebra, by \ref{enu:m2}, closed
under complex conjugation; and it separates points in $M$ by Corollary
\ref{cor:lap2}. \end{proof}
\begin{example}[The binary tree]
\label{exa:btree} Let $A=\left\{ 0,1\right\} $, and $M:=\prod_{\mathbb{N}}A$
the infinite Cartesian product, as a Cantor space. Set $V:=$ all
finite words:
\begin{equation}
V=\bigcup_{n\in\mathbb{N}}\left\{ \left(\alpha_{1},\alpha_{2},\cdots,\alpha_{n}\right)\:\big|\:\alpha_{i}\in\left\{ 0,1\right\} \right\} ;\label{eq:bt1}
\end{equation}
and set $l\left(\left(\alpha_{1},\alpha_{2},\cdots,\alpha_{n}\right)\right)=:n$. 

For $\omega=\left(\omega_{k}\right)_{1}^{\infty}\in M$, set 
\begin{equation}
\omega\big|_{n}:=\left(\omega_{1},\omega_{2},\cdots,\omega_{n}\right)\in V.\label{eq:bt2}
\end{equation}
For two points $\omega,\omega'\in M$, we shall need the number 
\begin{equation}
l\left(\omega\cap\omega'\right)=\sup\left\{ n\::\:\omega\big|_{n}=\omega'\big|_{n}\right\} .\label{eq:bt3}
\end{equation}

Let $r:\mathbb{N}\rightarrow\mathbb{R}_{+}$ be given such that 
\begin{equation}
r\left(\emptyset\right)=0,\quad\sum_{n\in\mathbb{N}}r\left(n\right)<\infty.\label{eq:bt4}
\end{equation}
For conductance function $c:E\rightarrow\mathbb{R}_{+}$, set 
\begin{equation}
c_{\alpha,\left(\alpha t\right)}=\frac{1}{r\left(l\left(\alpha\right)\right)},\quad\forall\alpha\in V,\: t\in\left\{ 0,1\right\} .\label{eq:bt5}
\end{equation}
One checks that, when (\ref{eq:bt4}) holds, then 
\[
\lim_{n,m\rightarrow\infty}R^{\left(c\right)}\left(\omega\big|_{n},\omega\big|_{m}\right)=0.
\]

Consider the graph $G_{2}=\left(V,E\right)$ where the edges are ``lines''
between $\alpha$ and $\left(\alpha t\right)$, where $t\in\left\{ 0,1\right\} $.
See Fig \ref{fig:bt1}.\end{example}
\begin{fact*}
With the settings above, the metric completion $\widetilde{R^{\left(c\right)}}$
w.r.t. the resistance metric on $V$ is as follows: For $\omega,\omega'\in M$
(see Fig \ref{fig:bt2}), 
\begin{equation}
\widetilde{R^{\left(c\right)}}\left(\omega,\omega'\right)=2\sum_{n=l\left(\omega\cap\omega'\right)}^{\infty}r\left(n\right).\label{eq:bt6}
\end{equation}
Let $\mathscr{H}$ be the corresponding energy-Hilbert space $\simeq$
the RKHS of $k_{c}$. For $\alpha\in V$, let $\delta_{\alpha}$ be
the Dirac-mass at the vertex point $\alpha$. Then
\begin{equation}
\left\Vert \delta_{\alpha}\right\Vert _{\mathscr{H}}^{2}=\frac{2}{r\left(l\left(\alpha\right)\right)}+\frac{1}{r\left(l\left(\alpha\right)-1\right)}\:.\label{eq:bt7}
\end{equation}
\end{fact*}
\begin{proof}
To see this, note that $\alpha$ has the three neighbors sketched
in Fig \ref{fig:bt1}, i.e., $\alpha^{*}$, $\left(\alpha0\right)$,
and $\left(\alpha1\right)$, where $\alpha^{*}$ is the one-truncated
word, 
\begin{equation}
\widetilde{R^{\left(c\right)}}\left(\omega,\omega'\right)=2\sum_{n=l\left(\omega\cap\omega'\right)}^{\infty}r\left(n\right).\label{eq:bt8}
\end{equation}

One checks that when (\ref{eq:bt4}) is assumed, then the conditions
in point \ref{enu:m3} of the theorem are satisfied. 
\end{proof}
\begin{figure}
\includegraphics[width=0.3\columnwidth]{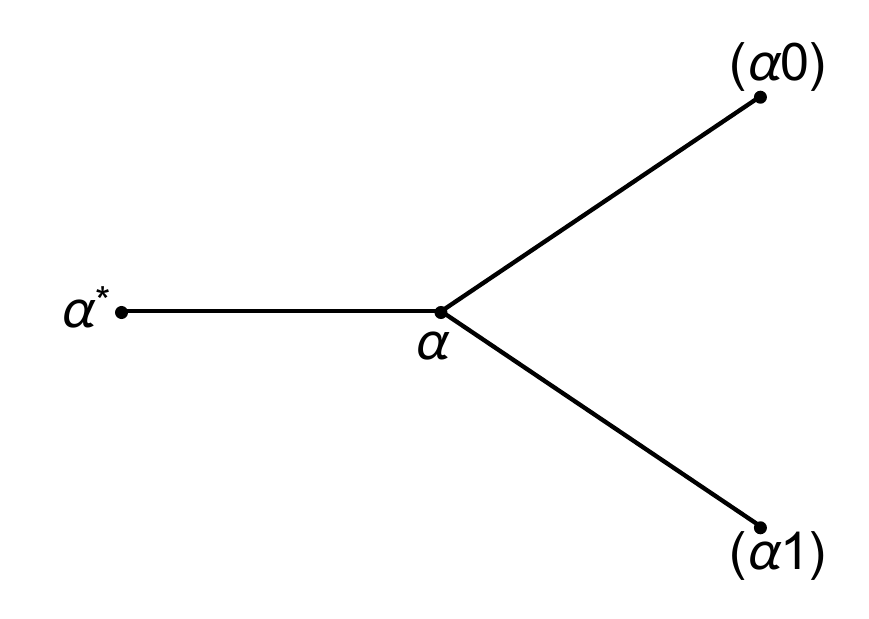}

\protect\caption{\label{fig:bt1}Edges in $G_{2}$.}
\end{figure}

\begin{figure}
\begin{tabular}{c}
\includegraphics[width=0.7\columnwidth]{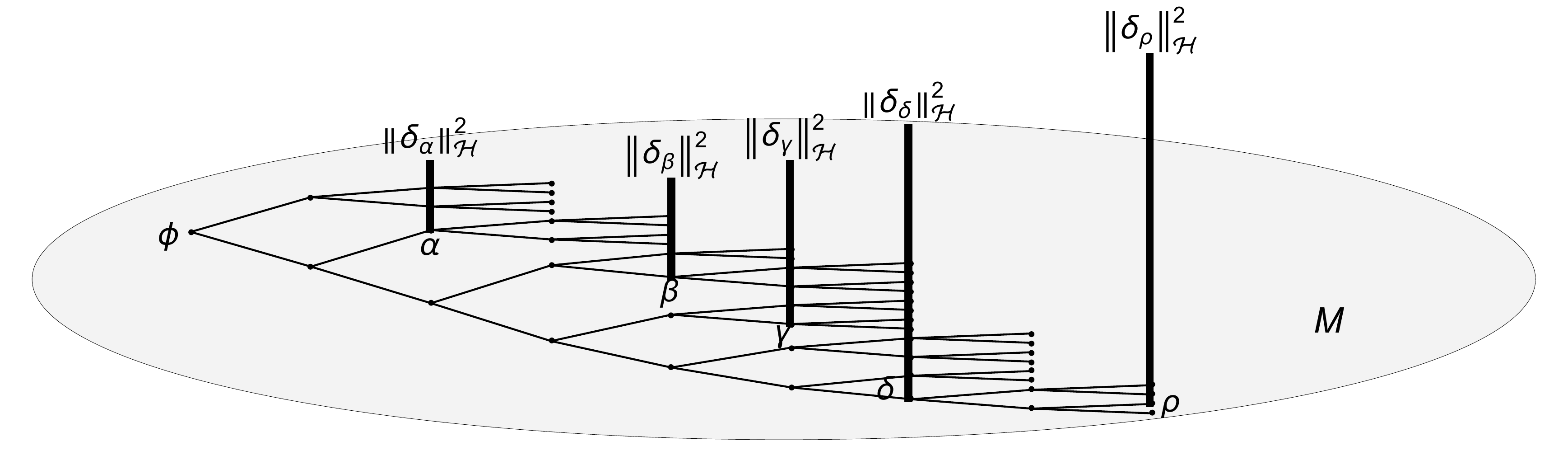}\tabularnewline
(a)\tabularnewline
\includegraphics[width=0.7\columnwidth]{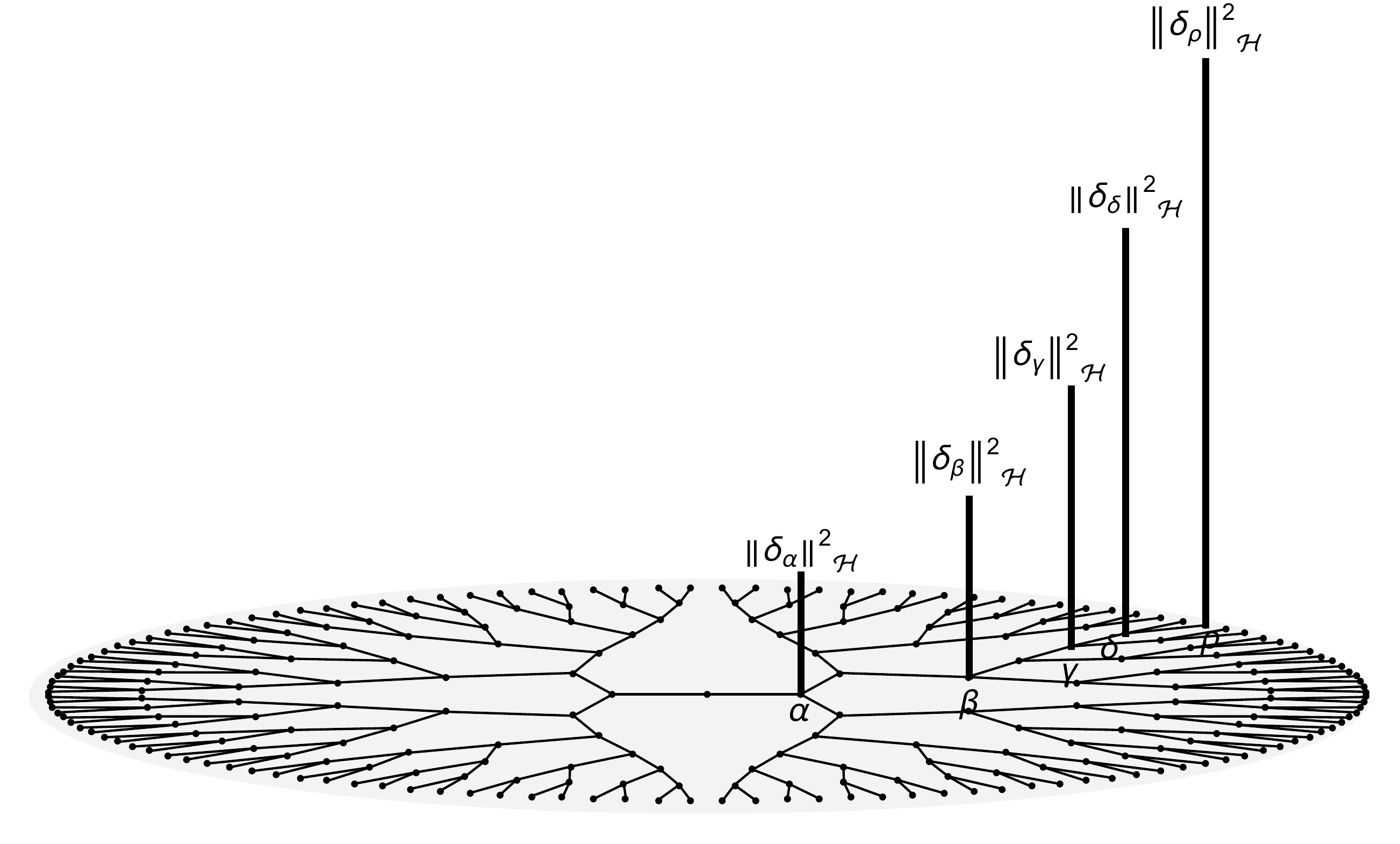}\tabularnewline
(b)\tabularnewline
\end{tabular}

\protect\caption{\label{fig:bt3}Histogram for $\left\Vert \delta_{\alpha}\right\Vert _{\mathscr{H}}^{2}$
as vertices $\alpha\in V$ approach the boundary. See (\ref{eq:bt7}),
and note $\left\Vert \delta_{\alpha}\right\Vert _{\mathscr{H}}^{2}\rightarrow\infty$
as $\alpha\rightarrow M$. }
\end{figure}

\begin{figure}
\includegraphics[width=0.7\columnwidth]{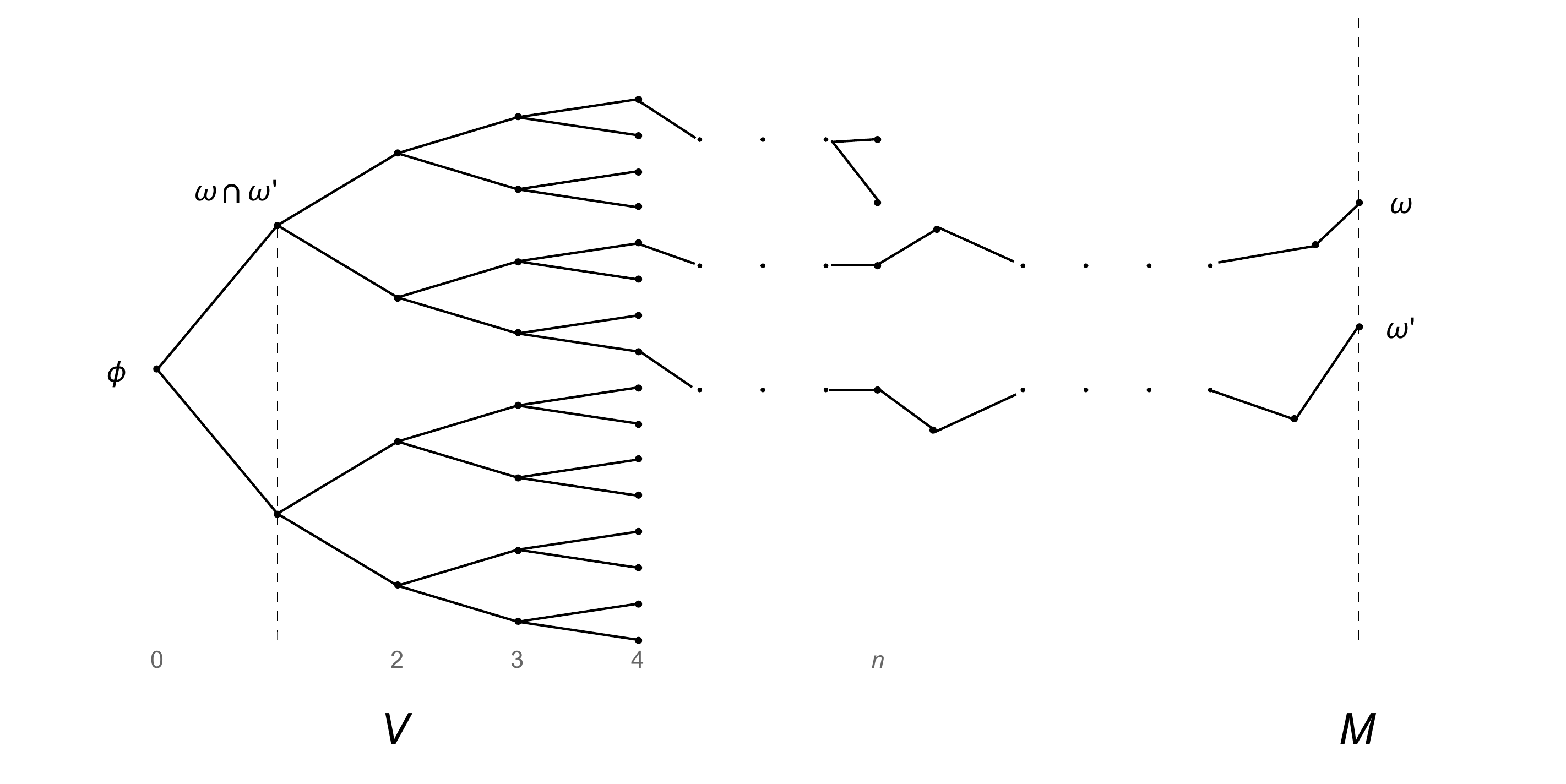}

\protect\caption{\label{fig:bt2}The binary tree and its boundary, the Cantor-set. }
\end{figure}

\begin{cor}
\label{cor:lap}Now return to the discrete restriction of Brownian
motion in sect. \ref{sub:bm}. Set $V=\left\{ x_{1},x_{2},x_{3},\cdots\right\} $
where the points $\left\{ x_{i}\right\} _{i=1}^{\infty}$ are prescribed
such that $x_{1}<x_{2}<\cdots<x_{i}<x_{i+1}<\cdots$. We turn $V$
into a weighted graph $G$ as follows: The edges $E$ in $G$ are
nearest neighbors; and we define a conductance function $c:E\rightarrow\mathbb{R}_{+}$
by setting 
\begin{equation}
c_{x_{i}x_{i+1}}:=\frac{1}{x_{i+1}-x_{i}},\label{eq:g13}
\end{equation}
and Laplace operator, 
\begin{equation}
\left(\Delta f\right)\left(x_{i}\right)=\frac{1}{x_{i+1}-x_{i}}\left(f\left(x_{i}\right)-f\left(x_{i+1}\right)\right)+\frac{1}{x_{i}-x_{i-1}}\left(f\left(x_{i}\right)-f\left(x_{i-1}\right)\right).\label{eq:g14}
\end{equation}

Then the RKHS associated with the Green's function of $\Delta$ in
(\ref{eq:g14}) agrees with that from the kernel construction in sect.
\ref{sub:bm}, i.e., the discrete Cameron-Martin Hilbert space.\end{cor}
\begin{proof}
Immediate from the previous Proposition and its corollaries. \end{proof}
\begin{acknowledgement*}
The co-authors thank the following colleagues for helpful and enlightening
discussions: Professors Daniel Alpay, Sergii Bezuglyi, Ilwoo Cho,
Ka Sing Lau, Paul Muhly, Myung-Sin Song, Wayne Polyzou, Gestur Olafsson,
Keri Kornelson, and members in the Math Physics seminar at the University
of Iowa.

\bibliographystyle{amsalpha}
\bibliography{ref}
\end{acknowledgement*}

\end{document}